\documentclass[a4paper, british]{amsart}

\usepackage{amssymb}
\usepackage{babel}
\usepackage{enumitem}
\usepackage{hyperref}
\usepackage[utf8]{inputenc}
\usepackage{newunicodechar}
\usepackage{varioref}
\usepackage[arrow,curve,matrix]{xy}

\usepackage{graphicx}
\usepackage{tikz}

\usepackage{mathpazo}
\usepackage{mathrsfs}
\usepackage{mathabx}

\numberwithin{figure}{section}

\usepackage[hyperpageref]{backref}

\setdescription{labelindent=\parindent, leftmargin=2\parindent}
\setitemize[1]{labelindent=\parindent, leftmargin=2\parindent}
\setenumerate[1]{labelindent=0cm, leftmargin=*, widest=iiii}

\DeclareFontFamily{OMS}{rsfs}{\skewchar\font'60}
\DeclareFontShape{OMS}{rsfs}{m}{n}{<-5>rsfs5 <5-7>rsfs7 <7->rsfs10 }{}
\DeclareSymbolFont{rsfs}{OMS}{rsfs}{m}{n}
\DeclareSymbolFontAlphabet{\scr}{rsfs}
\DeclareSymbolFontAlphabet{\scr}{rsfs}

\newunicodechar{α}{\alpha}
\newunicodechar{β}{\beta}
\newunicodechar{γ}{\gamma}
\newunicodechar{δ}{\delta}
\newunicodechar{Δ}{\Delta}
\newunicodechar{Γ}{\Gamma}
\newunicodechar{η}{\eta}
\newunicodechar{ι}{\iota}
\newunicodechar{κ}{\kappa}
\newunicodechar{λ}{\lambda}
\newunicodechar{Λ}{\Lambda}
\newunicodechar{μ}{\mu}
\newunicodechar{ω}{\omega}
\newunicodechar{Ω}{\Omega}
\newunicodechar{π}{\pi}
\newunicodechar{φ}{\phi}
\newunicodechar{Φ}{\Phi}
\newunicodechar{ψ}{\psi}
\newunicodechar{Ψ}{\Psi}
\newunicodechar{σ}{\sigma}
\newunicodechar{ρ}{\rho}
\newunicodechar{Σ}{\Sigma}
\newunicodechar{τ}{\tau}

\newunicodechar{→}{\to}
\newunicodechar{⨯}{\times}
\newunicodechar{∪}{\cup}
\newunicodechar{∩}{\cap}
\newunicodechar{⊇}{\supseteq}
\newunicodechar{⊃}{\supset}
\newunicodechar{⊆}{\subseteq}
\newunicodechar{⊂}{\subset}
\newunicodechar{≥}{\geq}
\newunicodechar{≤}{\leq}
\newunicodechar{∈}{\in}
\newunicodechar{◦}{\circ}
\newunicodechar{°}{^\circ}

\DeclareMathOperator{\codim}{codim}
\DeclareMathOperator{\Hom}{Hom}
\DeclareMathOperator{\Id}{Id}
\DeclareMathOperator{\Image}{Image}
\DeclareMathOperator{\reg}{reg}
\DeclareMathOperator{\sing}{sing}
\DeclareMathOperator{\Spec}{Spec}
\DeclareMathOperator{\supp}{supp}
\DeclareMathOperator{\tor}{tor}

\newcommand{\sA}{\scr{A}}
\newcommand{\sB}{\scr{B}}

\newcommand{\sF}{\scr{F}}
\newcommand{\sG}{\scr{G}}

\newcommand{\sJ}{\scr{J}}

\newcommand{\sO}{\scr{O}}

\newcommand{\sR}{\scr{R}}

\newcommand{\bA}{\mathbb{A}}

\newcommand{\bC}{\mathbb{C}}

\newcommand{\bP}{\mathbb{P}}
\newcommand{\bQ}{\mathbb{Q}}

\theoremstyle{plain}   
\newtheorem{thm}{Theorem}[section]
\newtheorem{aassumption}[thm]{Additional Assumption}

\newtheorem{consequence}[thm]{Consequence}
\newtheorem{cor}[thm]{Corollary}
\newtheorem{defn}[thm]{Definition} 

\newtheorem{lem}[thm]{Lemma}
\newtheorem{lemDef}[thm]{Lemma and Definition}
\newtheorem{lemNot}[thm]{Lemma and Notation}
\newtheorem{problem}[thm]{Problem}
\newtheorem{prop}[thm]{Proposition}

\theoremstyle{remark}
\newtheorem{asswlog}[thm]{Assumption without loss of generality}
\newtheorem{claim}[thm]{Claim}
\newtheorem{construction}[thm]{Construction}
\newtheorem{example}[thm]{Example}
\newtheorem{explanation}[thm]{Explanation}
\newtheorem{notation}[thm]{Notation}

\newtheorem{rem}[thm]{Remark}
\newtheorem{setting}[thm]{Setting}
\newtheorem{warning}[thm]{Warning}

\numberwithin{equation}{thm}

\newcommand{\into}{\hookrightarrow}

\newcommand{\wtilde}{\widetilde}
\newcommand{\what}{\widehat}

\def\clap#1{\hbox to 0pt{\hss#1\hss}}

\def\mathclap{\mathpalette\mathclapinternal}

\def\mathclapinternal#1#2{%
\clap{$\mathsurround=0pt#1{#2}$}}

\newcommand\CounterStep{\addtocounter{thm}{1}\setcounter{equation}{0}}

\def\factor#1.#2.{\left. \raise 2pt\hbox{$#1$} \right/\hskip -2pt\raise -2pt\hbox{$#2$}}

\newcommand{\PreprintAndPublication}[2]{%
  \sideremark{%
    \begin{color}{blue}Preprint\end{color}/%
    \begin{color}{green}Publication\end{color}}%
    \begin{color}{blue}#1\end{color}%
    \begin{color}{green}#2\end{color}%
    \sideremark{End of
    \begin{color}{blue}Preprint\end{color}/%
    \begin{color}{green}Publication\end{color}}}
\renewcommand{\PreprintAndPublication}[2]{#1}

\newcommand{\dK}{\ensuremath{d_\text{\rm Kähler}}}
\newcommand{\kahler}{\ensuremath{\text{\rm Kähler}}}
\newcommand{\dnoTor}{\ensuremath{d_\text{\rm tfree}}}
\newcommand{\drefl}{\ensuremath{d_\text{\rm refl}}}
\newcommand{\dtor}{\ensuremath{d_\text{\rm tor}}}

\DeclareMathOperator{\restr}{restriction}

\title{Pull-back morphisms for reflexive differential forms}
\date{\today}
\author{Stefan Kebekus}

\thanks{The author gratefully acknowledges partial support by the
  DFG-Forschergruppe 790 ``Classification of Algebraic Surfaces and Compact
  Complex Manifolds''.}

\address{Stefan Kebekus, Mathematisches Institut, Albert-Ludwigs-Universität
  Freiburg, Eckerstraße 1, 79104 Freiburg im Breisgau, Germany}
\email{\href{mailto:stefan.kebekus@math.uni-freiburg.de}{stefan.kebekus@math.uni-freiburg.de}}
\urladdr{\href{http://home.mathematik.uni-freiburg.de/kebekus}{http://home.mathematik.uni-freiburg.de/kebekus}}



\definecolor{linkred}{rgb}{0.7,0.2,0.2}
\definecolor{linkblue}{rgb}{0,0.2,0.6}
\hypersetup{
  pdfauthor={Stefan Kebekus},
  pdftitle={Pull-back morphisms for reflexive differential forms},
  pdfkeywords={Differential Form, Minimal Model Program, KLT Singularities},
  pdfstartview={Fit},
  pdfpagelayout={TwoColumnRight},
  pdfpagemode={UseOutlines},
  bookmarks,
  colorlinks,
  linkcolor=linkblue,
  citecolor=linkred,
  urlcolor=linkred
}

\begin{document}

\begin{abstract}
  Let $f : X → Y$ be a morphism between normal complex varieties, where $Y$ is
  Kawamata log terminal.  Given any differential form $σ$, defined on the smooth
  locus of $Y$, we construct a ``pull-back form'' on $X$.  The pull-back map
  obtained by this construction is $\sO_Y$-linear, uniquely determined by
  natural universal properties and exists even in cases where the image of $f$
  is entirely contained in the singular locus of $Y$.

  One relevant setting covered by the construction is that where $f$ is the
  inclusion (or normalisation) of the singular locus $Y_{\sing}$.  As an
  immediate corollary, we show that differential forms defined on the smooth
  locus of $Y$ induce forms on every stratum of the singularity stratification.
  The same result also holds for many Whitney stratifications.
\end{abstract}

\maketitle
\tableofcontents

\section{Introduction}

\subsection{Introduction and statement of main result}
\label{ssec:intro1}

Differential forms and sheaves of differentials are fundamental objects and
indispensable tools in the study of smooth varieties and complex manifolds.  It
is well-known that for singular spaces, there is no single notion of
``differential form'' that captures all features of the smooth case.  Instead,
there are several competing definitions, each generalising certain aspects.  The
following two classes of differential forms are particularly important.

\begin{description}
\item[Kähler differentials] Uniquely determined by universal properties, Kähler
  differentials are the most fundamental notion of differential form.  Their
  universal properties imply that Kähler differentials can be pulled back via
  arbitrary morphisms, and closely relate them to problems in deformation
  theory.
  
  The sheaf of Kähler differentials on a scheme $X$ is denoted $Ω^p_X$.  As a
  sheaf of $\sO_X$-modules, $Ω^p_X$ is usually rather complicated and has both
  torsion and cotorsion, even if $X$ is reduced.  This makes Kähler
  differentials rather hard to work with in many cases of practical interest.
  \smallskip

\item[Reflexive differentials] Given a normal complex variety $X$, a
  \emph{reflexive differential} on $X$ is a differential form defined only on
  the smooth locus, without imposing any boundary condition near the
  singularities.  Equivalently, a reflexive differential is a section in the
  double dual of the sheaf of Kähler differentials.  Denoting the sheaf of
  reflexive differentials by $Ω^{[p]}_X$, we have
  $$
  Ω^{[p]}_X = \bigl( Ω^p_X \bigr)^{**} = ι_* \bigl(
  Ω^p_{X_{\reg}} \bigr),
  $$
  where $ι: X_{\reg} → X$ denotes the inclusion of the smooth locus.  Since the
  $\sO_X$-module structure is comparatively simple, reflexive differentials are
  often quite useful in practise.  Reflexive differentials arise naturally in a
  number of contexts, for instance in positivity results for differentials on
  moduli spaces.  Note that the dualising sheaf is always a sheaf of reflexive
  differentials, $ω_X = Ω^{[\dim X]}_X$.
\end{description}

\subsubsection*{Reflexive differentials on klt spaces}

Reflexive differential forms do not generally satisfy the same universal
properties as Kähler differentials.  However, it has been shown in a previous
paper \cite{GKKP11} that many of the functorial properties do hold if we
restrict ourselves to normal spaces with Kawamata log terminal (klt)
singularities.  The following theorems summarises one of the main results in
this direction.

\begin{thm}[\protect{Extension Theorem, \cite[Thm.~1.4]{GKKP11}}]\label{thm:extpb}
  Let $Y$ be a normal variety and $f: X → Y$ a resolution of singularities.
  Assume that there exists a Weil divisor $D$ on $Y$ such that the pair $(Y,D)$
  is klt.  If
  $$
  σ ∈ H^0 \bigl( Y,\, Ω^{[p]}_Y \bigr) = H^0 \bigl( Y_{\reg},\,
  Ω^p_{Y_{\reg}} \bigr)
  $$
  is any reflexive differential form on $Y$, then there exists a differential
  form $τ ∈ H^0 \bigl( X,\, Ω^p_X \bigr)$ that agrees on the complement of the
  $f$-exceptional set with the usual pull-back of the Kähler differential
  $σ|_{Y_{\reg}}$.  \qed
\end{thm}

\begin{thm}[\protect{Pull-back morphisms in special cases, \cite[Thm.~4.3]{GKKP11}}]\label{thm:pbGKKP}
  Let $f: X → Y$ be a morphism of normal varieties.  Assume that there exists a
  Weil divisor $D$ on $Y$ such that the pair $(Y,D)$ is klt, and assume that the
  image of $f$ is not contained in the singular set $Y_{\sing}$.  Then there
  exist pull-back morphisms of reflexive forms,
  $$
  \drefl f : f^* Ω^{[p]}_Y → Ω^{[p]}_X \quad \text{and} \quad \drefl
  f : H^0 \bigl( Y,\, Ω^{[p]}_Y \bigr) → H^0 \bigl( X,\, Ω^{[p]}_X
  \bigr),
  $$
  that agrees with the usual pull-back of Kähler differentials wherever this
  makes sense.  \qed
\end{thm}

Theorems~\ref{thm:extpb} and its corollary, Theorem~\ref{thm:pbGKKP}, allow to
study reflexive differentials in the context of the minimal model program.
These results have been applied to a variety of settings, including a study of
hyperbolicity of moduli spaces, \cite{KK10} or see \cite{Keb11} for a survey, a
partial generalisation of the Beauville--Bogomolov decomposition theorem
\cite{GKP11} and deformations of Calabi--Yau varieties \cite{Kol12}.

\subsubsection*{Main result}

The main result asserts the existence of a useful pull-back morphism in a more
general setting.

\begin{thm}[Existence of pull-back morphisms in general, compare Theorem~\ref{thm:PB-thmA}]\label{thm:main}
  Let $f : X → Y$ be any morphism between normal complex varieties.  Assume that
  there exists a Weil divisor $D$ on $Y$ such that the pair $(Y,D)$ is klt.
  Then there exists a pull-back morphism
  $$
  \drefl f : f^* Ω^{[p]}_Y → Ω^{[p]}_X,
  $$
  uniquely determined by natural universal properties.
\end{thm}

\begin{rem}[Reference to precise statement]
  The ``natural universal properties'' mentioned in Theorem~\ref{thm:main} are a
  little awkward to formulate.  Precise statements are given in
  Theorem~\ref{thm:PB-thmA} and Section~\ref{ssec:mwad}.  In essence, it is
  required that the pull-back morphisms agree with the pull-back of Kähler
  differentials wherever this makes sense, and that they satisfy the composition
  law.
\end{rem}

\subsubsection*{Discussion of the main result}

It should be noted that Theorem~\ref{thm:main} does not require the image of $f$
to intersect the smooth locus of $Y_{\reg}$.  One particularly relevant setting
to which Theorem~\ref{thm:main} applies is that of a klt space $Y$, and the
inclusion (or normalisation) of the singular locus, say $f: X = Y_{\sing} → Y$.
It might seem surprising that a pull-back morphism exists in this context,
because reflexive differential forms on $Y$ are, by definition, differential
forms defined on the \emph{complement} of $Y_{\sing}$, and no boundary
conditions are imposed that would govern the behaviour of those forms near the
singularities.

In essence, Theorem~\ref{thm:main} asserts that differential forms defined on
the complement of $Y_{\sing}$ determine forms \emph{on} $Y_{\sing}$.  The
following immediate corollary gives a precise formulation.

\begin{cor}[Extension across the singularity stratification]\label{cor:singStrat}
  Let $Y$ be a normal complex variety with klt singularities.  Define a
  stratification of $Y$ by disjoint, locally closed, smooth subvarieties,
  $(Y^i)_{0 ≤ i ≤ k} ⊆ Y$, as follows.  Consider the chain of
  closed subvarieties,
  $$
  Y = \widehat Y^0 \supsetneq \widehat Y^1 \supsetneq \widehat Y^2 \supsetneq \cdots \supsetneq \widehat Y^k = \emptyset, \quad \text{  where $\widehat Y^{i+1} := \widehat Y^i_{\sing}$.}
  $$
  Given any index $i$, let $Y^i$ be the smooth locus of $\widehat Y^{i}$, that
  is, $Y^i := \widehat Y^{i}_{\reg}$.  Then, given any differential form
  $$
  σ_0 ∈ H^0 \bigl( Y^0,\, Ω^p_{Y^0} \bigr) = H^0 \bigl( Y_{\reg},\,
  Ω^p_{Y_{\reg}} \bigr),
  $$
  we obtain a sequence of induced forms $σ_i ∈ H^0 \bigl( Y^i,\,
  Ω^p_{Y^i} \bigr)$, for all $0 < i < k$.  \qed
\end{cor}

\begin{rem}
  The singularity stratification defined in Corollary~\ref{cor:singStrat} is the
  coarsest stratification whose strata are locally closed and smooth.  The
  conclusion also holds for finer stratifications, such as Whitney
  stratifications used in the discussion of intersection homology and perverse
  sheaves.
\end{rem}

\subsection{Optimality of the result}

The following two examples show that it is generally not possible to construct
reasonable pull-back morphisms for reflexive differentials on even the simplest
log canonical spaces.

\begin{setting}[Setting for Examples~\ref{ex:nlc1} and \ref{ex:nlc2}]\label{setting:ex12}
  Let $C ⊂ \bP^2$ be a smooth elliptic curve, and let $X ⊂ \bC^3$ be the affine
  cone over $E$.  An elementary computation shows that $X$ is normal and has an
  isolated, Gorenstein, log canonical singularity at the vertex point $x ∈ X$.
  The canonical bundle $ω_X = Ω^{[2]}_X$ is in fact even trivial.

  We denote the blow-up of the vertex point $x$ by $β : \wtilde X → X$.  The
  variety $\wtilde X$ is then smooth.  Denoting the $β$-exceptional curve by
  $E$, there exists an isomorphism $ω_{\wtilde X} \cong \sO_{\wtilde X}(-E)$.
  Observing that $\wtilde X$ is isomorphic to the total space of the line bundle
  $\sO_C(1)$, we have constructed the following commutative diagram of morphisms
  between normal varieties,
  \begin{equation}\label{diageset17}
    \xymatrix{ %
      E \ar[rrr]^{ι_E \text{, inclusion}} \ar[d]_{c \text{, constant}} &&& \wtilde X \ar[rrr]^{\text{$π$, $\bA^1$-bundle}} \ar[d]^{\text{$β$, blow-up}} &&& C \\
      \{x \} \ar[rrr]_{ι_x \text{, inclusion}} &&& X.
    }
  \end{equation}
\end{setting}

\begin{example}[Problems arising from differential forms with poles]\label{ex:nlc1}
  We maintain assumptions and notation of Setting~\ref{setting:ex12}.  To give a
  pull-back map that agrees outside of $E$ with the usual pull-back map for
  Kähler differentials, it is then equivalent to give a sheaf morphism
  $$
  \sO_{\wtilde X} \cong π^* ω_X \xrightarrow{\quad \drefl β \quad}
  ω_{\wtilde X} \cong \sO_{\wtilde X}(-E)
  $$
  that is isomorphic away from $E$.  Such a morphism does not exist.  If $σ ∈
  H^0\bigl( X,\, ω_X \bigr)$ denotes a global generator of the canonical sheaf,
  there is no section $\wtilde{σ} ∈ H^0\bigl( \wtilde X,\, ω_{\wtilde X} \bigr)$
  that agrees with $σ$ away from the exceptional set $E$.
\end{example}

\begin{example}[Impossibility to satisfy composition law]\label{ex:nlc2}
  We maintain assumptions and notation of Setting~\ref{setting:ex12}.  We
  consider reflexive one-forms on $X$ and assume there were pull-back morphisms
  $$
  \drefl β : H^0\bigl( X,\, Ω^{[1]}_{X} \bigr) → H^0\bigl( \wtilde
  X,\, Ω^1_{\wtilde X} \bigr) \quad \text{and} \quad \drefl ι_x :
  H^0\bigl( X,\, Ω^{[1]}_{X} \bigr) → \underbrace{H^0\bigl( \{x\} ,\,
  Ω^1_{\{x\}} \bigr)}_{= \{ 0\}}
  $$
  satisfying the following natural compatibility conditions.
  \begin{enumerate}
  \item\label{il:1} Away from the singular point, the map $\drefl β$ agrees with
    the usual pull-back map of Kähler differentials.
  \item\label{il:2} The composition law holds.  In particular,
    Diagram~\eqref{diageset17} will imply that
    $$
    \dK ι_E ◦ \drefl β = \dK c ◦ \drefl ι_x
    $$
    where $\dK$ denotes the usual pull-back of Kähler differentials.  Since $c$
    is the constant map, this implies that $\dK ι_E ◦ \drefl β = 0$.
  \end{enumerate}

  Let $τ_C ∈ H^0\bigl( C,\, Ω^1_C \bigr)\setminus \{0\}$ be any non-vanishing
  differential form on the elliptic curve $C$, and let $\wtilde{τ} := \dK π (τ)
  ∈ H^0\bigl( \wtilde X,\, Ω^1_{\wtilde X} \bigr)$ be its pull-back to $\wtilde
  X$.  Since $\wtilde X \setminus E$ and $X_{\reg} = X \setminus \{x \}$ are
  isomorphic, the form $\wtilde{τ}$ induces a (reflexive) differential form on
  $X$, say
  $$
  τ ∈ H^0\bigl( X,\, Ω^{[1]}_{X} \bigr) = H^0\bigl( X_{\reg},\,
  Ω^1_{X_{\reg}} \bigr).
  $$
  Property~(\ref{ex:nlc2}.\ref{il:1}) then implies that $\drefl β (τ) =
  \wtilde{τ}$, and
  $$
  \bigl( \dK ι_E ◦ \drefl β \bigr)(τ) = \dK ι_E (\wtilde{τ})
  \underset{\text{(\ref{ex:nlc2}.\ref{il:2})}}{=} \dK (\underbrace{π ◦
    ι_E}_{\mathclap{\text{isomorphism}}} ) (τ_C) \not = 0.
  $$
  This clearly contradicts Property~(\ref{ex:nlc2}.\ref{il:2}), showing that
  pull-back morphisms satisfying these compatibility conditions cannot exist.
\end{example}

In the setting of Example~\ref{ex:nlc1}, there does exist a differential form
with logarithmic poles along $E$, say $\wtilde{σ}∈ H^0\bigl( \wtilde X,\,
ω_{\wtilde X} \otimes \sO_{\wtilde X}(E) \bigr)$, that agrees with $σ$ away from
the singular set.  One could argue that $\wtilde{σ}$ should be taken as a
pull-back of $σ$.  While this might be a viable definition when discussing the
blow-up morphism $β$, problems occur as soon as one wishes to pull-back $σ$ via
the composition $ι_E ◦ β$.

In addition to the problems originating from the existence of poles,
Example~\ref{ex:nlc2} shows that there are other and more fundamental reasons
why reasonable pull-back maps cannot be defined for log canonical varieties:
there is in general no way to define a pull-back map in a way that is compatible
with the usual composition law.

\subsection{Idea of proof and outline of paper}
\label{sssec:idea}

The proof of our main result is technically somewhat involved.  The main idea,
however, is quite elementary and straightforward.  Consider the following simple
setting.

\begin{setting}[Setting for Examples~\ref{ex:pb1} and \ref{ex:pb2}]\label{setting:ex11}
  Let $Y$ be a variety with klt singularities.  Assume that the singular locus
  is a smooth curve $C$ and that the singularities of $Y$ are resolved by a
  single blow-up of the curve $C$, say $π : \wtilde Y → Y$.  Assume further that
  the exceptional set $E$ of this resolution map is irreducible and smooth over
  $C$.

  Let $σ ∈ H^0 \bigl( Y,\, Ω^{[1]}_Y \bigr)$ be any given reflexive differential
  form on $Y$.  Theorem~\ref{thm:extpb} asserts the existence of a regular
  differential form $\wtilde{σ} ∈ H^0 \bigl( \wtilde Y,\, Ω^1_{\wtilde Y}
  \bigr)$ that agrees outside of $E$ with the form $σ$.
\end{setting}

\begin{example}[Pulling back reflexive differentials via a resolution map]\label{ex:pb1}
  In Setting~\ref{setting:ex11}, define $\wtilde{σ}$ as the pull-back of the
  form $σ$ to $\wtilde Y$.  This choice is unique if we require our
  pull-back form to agree on the smooth locus with the usual pull-back of Kähler
  differentials.
\end{example}

\begin{example}[Pulling back reflexive differentials in more generality]\label{ex:pb2}
  In Setting~\ref{setting:ex11}, let $X$ be a smooth variety and $f : X → Y$ a
  morphism whose image is contained in $C$.  We aim to define a pull-back form
  $σ_X ∈ H^0 \bigl( X,\, Ω^1_X \bigr)$.

  A fundamental theorem of Hacon and McKernan, \cite[Cor.~1.5]{HMcK07}, asserts
  that the fibres of $π|_E : E → C$ are rationally connected manifolds.
  Recalling that rationally connected manifolds do not admit non-trivial
  differential forms, the long exact sequence of the relative differential
  sequence,
  $$
  0 → H^0\bigl( C,\, Ω^1_C \bigr) \xrightarrow{\quad d(π|_E) \quad} H^0
  \bigl( E,\, Ω^1_E \bigr) → \underbrace{H^0 \bigl( E,\, Ω^1_{E/C}
    \bigr)}_{= 0} → \cdots,
  $$
  shows that the restriction of $\wtilde{σ}$ to $E$ really is the pull-back of a
  form $τ ∈ H^0 \bigl( C,\, Ω^1_C \bigr)$ on $C$.  Let $σ_X$ be the standard
  pull-back of the Kähler differential $τ$ to $X$.

  The choice of $σ_X$ is unique if we require that the pull-back morphisms
  satisfy the composition law.  To this end, recall Graber--Harris--Starr's
  generalisation of Tsen's theorem, \cite[Thm.~1.1]{GHS03}, which gives the
  existence of a section $s : C → E ⊂ E$, forming a commutative diagram
  $$
  \xymatrix{ %
    && E \ar@{^(->}[rr]^{ι_E, \text{ inclusion}} && \wtilde Y \ar[d]^{π, \text{ resolution}} \\
    X \ar[rr]^{f_C} \ar@/_3mm/[rrrr]_f && C \ar[u]^{s} \ar@{^(->}[rr]^{ι_C, \text{ inclusion}} &&  Y 
  }
  $$
  Choosing one such $s$, observe that
  \begin{align*}
    σ_X & = (\dK f_C) (\dK s) (\dK ι_E) (\dK π)(σ)  \\
    & = (\dK f_C) (\dK s) (\dK ι_E) (\wtilde{σ}) && \text{by Example~\ref{ex:pb1}} \\
    & = (\dK f_C) (\dK s) (\wtilde{σ}|_E) \\
    & = (\dK f_C) (τ).
  \end{align*}
  Note that this completely dictates the choice of $σ_X$.
\end{example}

Two problems occur when trying to adapt the ideas outlined in
Examples~\ref{ex:pb1}--\ref{ex:pb2} to the general setting, where $X$ is allowed
to have arbitrary klt singularities.

\begin{problem}\label{prob:1}
  The exceptional set of a resolution morphism need not be irreducible, or
  relatively smooth over the singular locus.  Fibres of resolution maps are known
  to be rationally chain connected, but Kähler differentials might well exist on
  these non-normal spaces.
\end{problem}

\begin{problem}\label{prop:2}
  The result of Graber--Harris--Starr is specific to families over 1-dimensional
  base varieties, and is not known to hold in higher dimensions.
\end{problem}

To overcome the Problem~\ref{prob:1}, we need to consider the sheaves
$\widecheck{Ω}^p$ of ``Kähler differentials modulo torsion'' and discuss their
properties on reduced, reducible, and not necessarily normal schemes.  This is
done in Part~\ref{part:1}.  There, we establish a number of fundamental
universal properties and show that reduced, reducible, rationally chain
connected schemes with simple normal crossings do not admit any ``Kähler
differential modulo torsion''.  The notions of torsion and torsion-free sheaves
on reducible spaces do not seem be discussed much in the literature.  For the
reader's convenience, we recall the definition and establish basic properties in
Appendix~\ref{app:A}.

The Problem~\ref{prop:2} does not pose fundamental difficulties.  However, it
does make the proof of our main theorem, given in Part~\ref{part:2} of this
paper, somewhat awkward and lengthy as we constantly need to switch between the
spaces in question and suitable coverings, for which a section $s$ exists.

\PreprintAndPublication{}{The preprint version of this paper, available as
  \href{http://arxiv.org/abs/1210.3255}{arXiv:1210.3255} includes more details,
  explicit computations and several additional figures.}

\subsection{Notation and global conventions}

Throughout the text, we work over the complex number field.  We discuss several
notions of differential forms on singular spaces, and in each case define
pull-back morphisms.  To avoid the obvious potential for confusion, we clearly
distinguish between the various notions of pulling back.

\begin{notation}[Sheaves of differentials and pull-back morphisms]
  If $f : X → Y$ is any morphism between varieties, we denote the sheaves of
  Kähler differentials by $Ω^p_X$ and $Ω^p_Y$, respectively.  The
  standard pull-back maps of Kähler differentials are denoted as
  $$
  \dK f : f^* Ω^p_Y → Ω^p_X \quad \text{and} \quad \dK f : H^0
  \bigl( Y,\, Ω^p_Y \bigr) → H^0 \bigl( X,\, Ω^p_X \bigr).
  $$

  Part~\ref{part:1} discusses the sheaf of Kähler differentials modulo torsion,
  traditionally denoted as $\widecheck{Ω}^p_X := Ω^p_X / \tor$.  As we will see
  in Section~\ref{sec:TFD}, there exists a meaningful notion of pulling back
  sections is this sheaf.  The associated morphisms of sheaves and vector spaces
  are denoted as
  $$
  \dnoTor f : f^* \widecheck{Ω}^p_Y → \widecheck{Ω}^p_X \quad \text{and} \quad
  \dnoTor f : H^0 \bigl( Y,\, \widecheck{Ω}^p_Y \bigr) → H^0 \bigl( X,\,
  \widecheck{Ω}^p_X \bigr).
  $$

  Part~\ref{part:2} discusses reflexive differentials.  The sheaf of reflexive
  differentials will always be denoted by $Ω^{[p]}_X := \bigl( Ω^p_X
  \bigr)^{**}$.  Following the notation introduced in Theorem~\ref{thm:pbGKKP},
  pull-back morphisms for reflexive differentials will be denoted by $\drefl f$.
\end{notation}

\begin{notation}[Sheaves defined on subschemes]\label{not:sdos}
  Let $X$ be a scheme and $Y ⊆ X$ a subscheme, with associated inclusion
  map $ι : Y → X$.  Sheaves $\sF$ on $Y$ will often be viewed as sheaves on
  the ambient space $X$.  If no confusion is likely to arise, we follow standard
  notation and write $\sF$ as a shorthand for the technically more correct
  $ι_* \sF$.
\end{notation}

\subsection{Acknowledgements}

The author would like to thank the members of the algebraic and arithmetic
geometry groups at the University of Freiburg for many inspiring
discussions.  Example~\ref{ex:nlc2} was perceived in discussion with Patrick
Graf.  Fritz Hörmann suggested a simplified proof of
Theorem~\ref{prop:pull-back-TFD1}.  Fred Rohrer (Tübingen) and Angelo Vistoli
(Bologna) kindly answered questions and provided references.

After the paper appeared in preprint form, we have learned from Clemens Jörder
that parts of Section~\ref{sec:TFD} overlap with \cite{Fer70}, which discusses
torsion-free differentials in a related settung, but using rather different
language.

\part{Torsion-free differentials}
\label{part:1}

\section{Torsion-free differentials and their pull-back properties}
\label{sec:TFD}

\subsection{The definition of torsion-free differentials}

As indicated in the introduction, we need to discuss sheaves of ``Kähler
differentials modulo torsion''.  The following notation will be used.  We refer
to Appendix~\ref{app:A} for a discussion of torsion sheaves on reduced, but
possibly reducible spaces. 

\begin{notation}[Torsion differentials and torsion-free differentials]
  Let $X → Y$ be a morphism of reduced, quasi-projective schemes.  Given any
  number $p$, define $\widecheck{Ω}^p_{X/Y}$ as the cokernel of the sequence
  \begin{equation}\label{eq:torSeq}
    \xymatrix{ %
      0 \ar[r] & \tor Ω^p_{X/Y} \ar[r]^{α_{X/Y}} & Ω^p_{X/Y} \ar[r]^{β_{X/Y}} & \widecheck{Ω}^p_{X/Y} \ar[r] & 0 %
    }
  \end{equation}
  If $Y$ is a point, we follow the usual notation and write $\widecheck{Ω}^p_X$
  instead of $\widecheck{Ω}^p_{X/Y}$.  Consider the torsion subsheaf $\tor
  Ω^p_{X/Y} ⊆ Ω^p_{X/Y}$, as introduced in Definition~\vref{defn:torsion}.
  Sections in $\tor Ω^p_{X/Y}$ and $\tor Ω^p_X$ are called (relative)
  \emph{torsion differentials}.  By slight abuse of language, we refer to
  sections in $\widecheck{Ω}^p_{X/Y}$ and $\widecheck{Ω}^p_X$ as (relative)
  \emph{torsion-free differentials}.
\end{notation}

\begin{rem}[Characterisation of torsion and torsion-free differentials]\label{gen:charTFD}
  Torsion differentials are characterised among Kähler differentials as those
  that vanish at general points of all irreducible components of $X$.  A
  torsion-free differential on $X$ vanishes if and only if it vanishes
  generically on all irreducible components of $X$.  We refer to
  Explanation~\ref{expl:torgenpts} for further discussion.
\end{rem}

\begin{notation}[Morphisms $α_{X/Y}$ and $β_{X/Y}$]
  Sequence~\eqref{eq:torSeq} is obviously of great importance in the discussion
  of torsion-free differentials.  We will therefore maintain the meaning of the
  symbols $α_{X/Y}$ and $β_{X/Y}$ throughout the present
  Section~\ref{sec:TFD}.  Again, if $Y$ is a point, we write $α_X$ instead
  of $α_{X/Y}$.
\end{notation}

\subsection{Pull-back properties}

Given a morphism between two varieties, we aim to show that the usual pull-back
map of Kähler differentials always induces a pull-back map of torsion-free
differentials, even if the image of the morphism is contained in the singular
set of the target variety.  The following proposition, which asserts that the
pull-back of a torsion-differential is always torsion, is a first step in this
direction.

\begin{prop}[Pull-back of torsion differentials are torsion]\label{prop:pull-back-TFD1}
  Let $f : X → Y$ be a morphism of reduced, quasi-projective schemes.  If $σ ∈
  H^0\bigl( Y,\, \tor Ω^p_Y \bigr)$ is a torsion-form on $Y$, then $\dK f (σ)$
  is a torsion form on $X$.  In other words,
  $$
  \dK f (σ) ∈ H^0\bigl( X,\, \tor Ω^p_X \bigr) ⊂ H^0\bigl( X,\, Ω^p_X \bigr).
  $$
\end{prop}

\begin{rem}
  The assertion of Proposition~\ref{prop:pull-back-TFD1} is clearly true if $X$
  is irreducible and $f(X)$ intersects the smooth locus of $Y$.  We claim that
  Proposition~\ref{prop:pull-back-TFD1} also holds in cases where $f(X)$ is
  entirely contained in the singular locus.  One particularly relevant case is
  the inclusion of the singular locus, say $X := Y_{\sing} \into Y$.  If
  $Y_{\sing}$ is itself smooth, then Proposition~\ref{prop:pull-back-TFD1}
  asserts that the pull-back of any torsion differential vanishes.
  \PreprintAndPublication{A simple case of this configuration is shown in
    Figure~\vref{fig:restTorDiff}.}{}
\end{rem}

\PreprintAndPublication{
\begin{figure}
  \centering

  $$
  \qquad X = \{ (x,y,z) ∈ \bC^3 : z^3+y^2-x^2z^2 = 0\}
  $$
  \includegraphics[width=8.0cm]{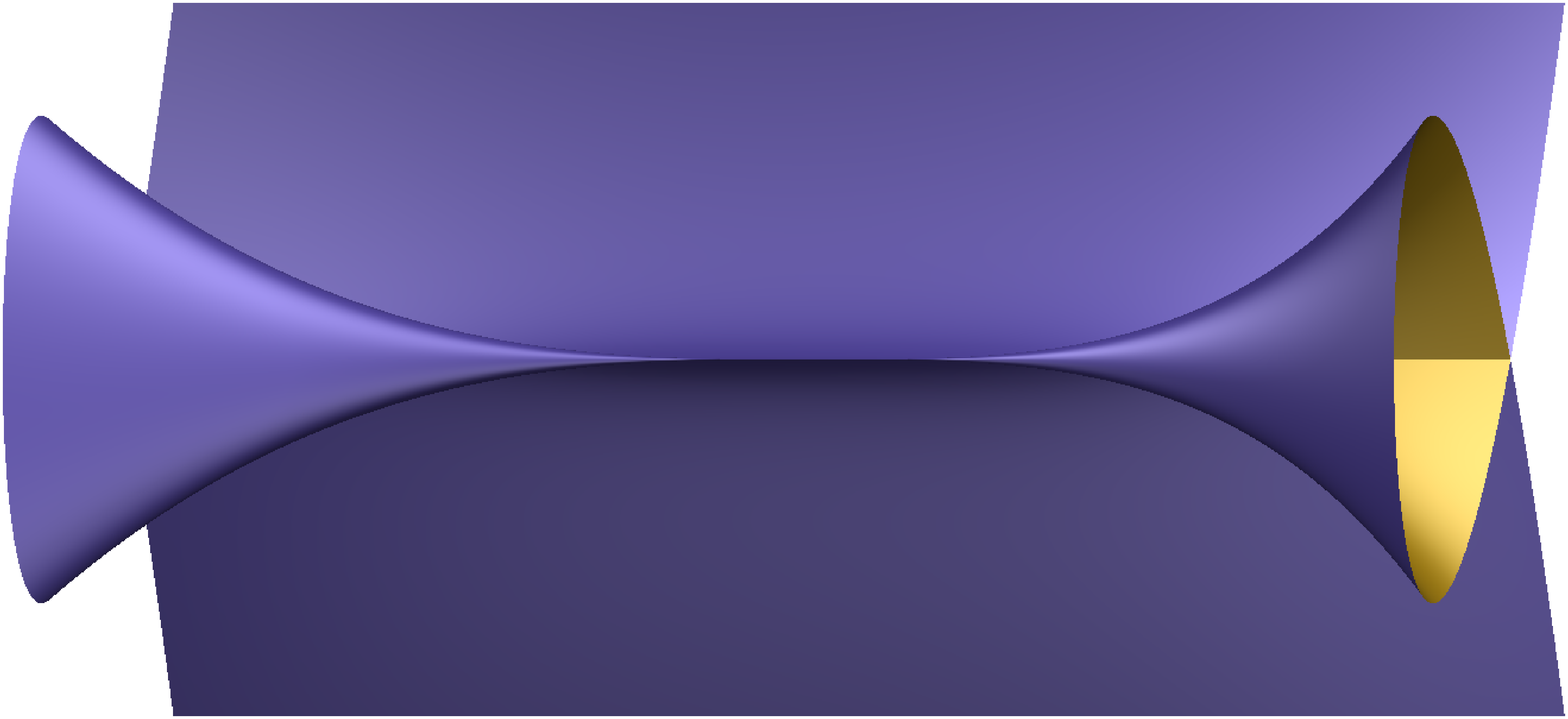}

  \bigskip

  {\small The figure sketches a special case of
    Proposition~\ref{prop:pull-back-TFD1}.  Here, it can be shown by elementary
    computation that the restriction of any torsion differential on $X$ to the
    singular locus $X_{\sing} = \{ z=y=0 \}$ vanishes.  To be more precise, if
    $ι: X_{\sing} → X$ denotes the inclusion map, then $(\dK ι)|_{\tor Ω^p_X} =
    0$.}

  \caption{Restriction of torsion differentials to the singular locus}
  \label{fig:restTorDiff}
\end{figure}
}{}

\begin{warning}[Proposition~\ref{prop:pull-back-TFD1} is wrong in the relative setup]
  The proof of Proposition~\ref{prop:pull-back-TFD1} relies on the existence of
  a desingularisation map for which no analogue exists in the relative setting.
  As a matter of fact, Proposition~\ref{prop:pull-back-TFD1} becomes wrong when
  working with relative differentials, unless one makes rather strong additional
  assumptions.  For a simple example, consider a sequence of morphisms,
  $$
  \xymatrix{ %
    X \ar[rrrr]^{f}_{\text{inclusion of except.  curve}} &&&& Y \ar[rrrr]^{π}_{\text{blow-up of smooth surface}} &&&& Z.
  }
  $$
  A simple computation shows that $Ω^1_{Y/Z} = \tor Ω^1_{Y/Z} = f_* Ω^1_X$.  In
  particular, if $U ⊂ Y$ is any open set and $σ ∈ H^0 \bigl( U,\, Ω^1_{Y/Z}
  \bigr)$ any non-vanishing differential on $U$, then $σ$ is torsion on $U$, but
  $\dK f(σ)$ is a torsion-free differential on the curve $f^{-1}(U)$.
\end{warning}

Before giving a proof of Proposition~\ref{prop:pull-back-TFD1} in
Section~\ref{ssec:pfpbTFD1}, we note a number of corollaries which will be
relevant in the further discussion.  Among these, the existence of a pull-back
map for torsion-free differentials is the most important.

\begin{cor}[Pull-back for sheaves of torsion-free differentials]\label{cor:pull-back-TFD1}
  Let $f : X → Y$ be a morphism of reduced, quasi-projective schemes.  Then
  there exist unique morphisms $\dtor f$ and $\dnoTor f$ such that the following
  diagram with exact rows becomes commutative
  \begin{equation}\label{eq:sqQ}
    \begin{split}
      \xymatrix{ %
        & f^* \tor Ω^p_{Y} \ar[rr]^{f^* α_{Y} } \ar[d]_{\dtor f} && f^* Ω^p_{Y} \ar[rr]^{f^* β_{Y}} \ar[d]_{\dK f} && f^* \widecheck{Ω}^p_{Y} \ar[r] \ar[d]^{\dnoTor} & 0 \\
        0 \ar[r] & \tor Ω^p_{X} \ar[rr]_{α_{X}} && Ω^p_{X} \ar[rr]_{β_{X}} && \widecheck{Ω}^p_{X} \ar[r] & 0.
      }
    \end{split}
  \end{equation}
\end{cor}
\begin{proof}
  Proposition~\ref{prop:pull-back-TFD1} immediately implies that the restricted
  morphism $(\dK f)|_{f^*\tor Ω^p_{Y}}$ factorises via the torsion subsheaf
  $\tor Ω^p_{X}$; this is clearly the unique morphism $\dtor f$ that makes the
  left square of \eqref{eq:sqQ} commutative.  Existence and uniqueness of
  $\dnoTor f$ follows from surjectivity of $f^* β_{Y}$ once we note that
  $$
  β_{X} ◦ (\dK f) ◦ f^* α_{Y} = 0.
  $$
  This finishes the proof of Corollary~\ref{cor:pull-back-TFD1}.
\end{proof}

The morphism $f^* β_{Y}$ in \eqref{eq:sqQ} is surjective.  This has a number of
immediate consequences.  We mention two which will be relevant in the further
discussion.

\begin{lem}\label{lem:dandcheckd}
  In the setting of Corollary~\ref{cor:pull-back-TFD1}, if $x ∈ X$ is any
  point such where $\dK f$ vanishes, then $\dnoTor f$ vanishes as well.  In other
  words, 
  $$
  (\dK f)|_{\{x\}} = 0 \,\, \Rightarrow \,\, (\dnoTor f)|_{\{x\}} = 0.  \makeatletter\displaymath@qed
  $$
\end{lem}

\begin{lem}[Composition law for pull-back of sheaves of torsion-free differentials]\label{lem:clTFD}
  In the setting of Corollary~\ref{cor:pull-back-TFD1}, let $g : X' → X$ be
  any morphism from a reduced, quasi-projective scheme $X'$.  Then
  \begin{equation}\label{eq:clTFD}
    \dnoTor (f ◦ g) = (\dnoTor g) ◦ g^*(\dnoTor f).
  \end{equation}
\end{lem}
\begin{proof}
  Consider the commutative diagram of sheaf morphisms,
  \begin{equation}\label{eq:cd-comp}
    \begin{split}
      \xymatrix{ %
        g^* f^* Ω^p_{Y} \ar[rr]^{g^* (\dK f)} \ar[d]_{g^* f^* β_{Y}} &&
        g^* Ω^p_{X} \ar[rr]^{\dK g} \ar[d]^{g^* β_{X}} &&
        Ω^p_{X'} \ar[d]^{β_{X'}} \\
        g^* f^* \widecheck{Ω}^p_{Y} \ar[rr]_{g^* (\dnoTor f)} &&
        g^* \widecheck{Ω}^p_{X} \ar[rr]_{\dnoTor g} &&
        \widecheck{Ω}^p_{X'}
      }
    \end{split}
  \end{equation}
  obtained by composing the right commutative square of Diagram~\eqref{eq:sqQ}
  for the morphism $g$ with the $g$-pull-back of the commutative square for the
  morphism $f$.  \PreprintAndPublication{Using the composition law for the
    pull-back of Kähler differentials,
    $$
    \dK (f ◦ g) = (\dK g) ◦ g^* (\dK f),
    $$
    the outer square of~\eqref{eq:cd-comp} is thus written as
    \begin{equation}\label{eq:cd-comp2}
      \begin{split}
        \xymatrix{ %
          g^* f^* Ω^p_{Y} \ar[rrr]^{\dK(f ◦ g)} \ar[d]_{g^* f^* β_{Y}} &&&
          Ω^p_{X'} \ar[d]^{β_{X'}} \\
          g^* f^* \widecheck{Ω}^p_{Y} \ar[rrr]_{(\dnoTor g) ◦ g^* (\dnoTor f)} &&&
          \widecheck{Ω}^p_{X'}
        }
      \end{split}
    \end{equation}
    As the pull-back of a surjective sheaf morphism, the map $g^* f^* β_{Y}$
    is surjective itself.  A comparison of Diagram~\eqref{eq:cd-comp2} with the
    right square of Diagram~\eqref{eq:sqQ} for the composed morphism $f ◦ g$
    thus immediately shows Equation~\eqref{eq:clTFD}, as claimed.  }{Using the
    composition law for the pull-back of Kähler differentials and surjectivity
    of $g^* f^* β_{Y}$, the claim follows from a comparison of the outer
    square of~\eqref{eq:cd-comp} with the right square of \eqref{eq:sqQ} for the
    composed morphism $f ◦ g$.}
\end{proof}

\begin{notation}[Pull-back for globally defined torsion-free differentials]
  In the setting of Corollary~\ref{cor:pull-back-TFD1}, the sheaf morphism
  $\dnoTor f$ induces a morphism between vector spaces of globally defined
  torsion-free forms, which we will again denote by $\dnoTor f : H^0 \bigl(Y,\,
  \widecheck{Ω}^p_Y \bigr) → H^0 \bigl(X,\, \widecheck{Ω}^p_X \bigr)$.
\end{notation}

\begin{rem}[Composition law for globally defined torsion-free differentials]
  In the setting of Corollary~\ref{cor:pull-back-TFD1}, given given a further
  morphism $g : X' → X$, Lemma~\ref{lem:clTFD} implies
  \PreprintAndPublication{that the following diagram is commutative,
    $$
    \xymatrix{ %
      H^0 \bigl(Y,\, \widecheck{Ω}^p_Y \bigr) \ar[rr]_{\dnoTor f}
      \ar@/^5mm/[rrrr]^{\dnoTor (f◦ g)} && H^0 \bigl(X,\, \widecheck{Ω}^p_X
      \bigr) \ar[rr]_{\dnoTor g} && H^0 \bigl(X',\, \widecheck{Ω}^p_{X'} \bigr).  }
    $$
  }{the following composition law for pull-back morphisms of globally defined
    forms, $\dnoTor (f◦ g) = \dnoTor g ◦ \dnoTor f$.}
\end{rem}

\subsection{Proof of Proposition~\ref*{prop:pull-back-TFD1}}
\label{ssec:pfpbTFD1}

We are grateful to Fritz Hörmann for help with the following proof.  We maintain
notation and assumptions of Proposition~\ref{prop:pull-back-TFD1} and assume
that $σ$ is a torsion form on $Y$.

\subsubsection{Simplifications}

Recall from Remark~\ref{gen:charTFD} or Explanation~\ref{expl:torgenpts} that to
show that $\dK f(σ)$ is a torsion differential, it suffices to show that it
vanishes at general points of each irreducible component of $X$.  This simple
observation has two important consequences.

First consequence: considering one component of $X$ at a time, we are free to
assume without loss of generality that $X$ is irreducible.  Replacing $X$ by a
suitable dense open subset, we can even assume the following.

\begin{asswlog}\label{ass:211}
  The scheme $X$ is irreducible and smooth.
\end{asswlog}

Second consequence: Let $Y_0 ⊆ Y$ be an irreducible component that contains the
image $f(X)$.  The morphism $f$ factors via the inclusion map,
$$
\xymatrix{
  X \ar[rr]_{f_0} \ar@/^3mm/[rrrr]^{f} && Y_0 \ar[rr]_{ι_0} && Y,
}
$$
and $\dK f(σ) = (\dK f_0) \bigl(\dK ι_0(σ) \bigr)$.  Since $σ$ is torsion,
$\dKι_0(σ)$ vanishes at general points of $Y_0$, and is therefore a torsion form
on $Y_0$.  Proposition~\ref{prop:pull-back-TFD1} will therefore follow for our
given morphism $f$ if we can prove it for $f_0$.  We can therefore assume
without loss of generality that the following holds.

\begin{asswlog}\label{ass:212}
  The scheme $Y$ is irreducible.   
\end{asswlog}

\subsubsection{End of proof}

Let $π_Y : \wtilde Y → Y$ be a resolution of singularities which exists because
$Y$ is irreducible.  Choose a component of $\wtilde X ⊆ X ⨯_Y \wtilde Y$ that
surjects onto $X$, and let $\wtilde X$ be a desingularisation of that component.
We obtain a commutative diagram,
$$
\xymatrix{ %
  \wtilde X \ar[r]^{\wtilde f} \ar[d]_{π_X} & \wtilde Y \ar[d]^{π_Y} \\
  X \ar[r]_{f} & Y.
}
$$

Since $σ$ vanishes at general points of $Y$ and since $π_Y$ is birational, it is
clear that the pull-back $\dKπ_Y(σ)$ vanishes at general points of $\wtilde Y$.
But since $\wtilde Y$ is smooth by Assumption~\ref{ass:211}, this means that it
vanishes everywhere, \CounterStep
\begin{equation}\label{eq:aa}
  \dK π_Y(σ) = 0 ∈ H^0\bigl( \wtilde Y,\, Ω^p_{\wtilde Y}\bigr).
\end{equation}

Since $Ω_X$ is torsion-free and $π_X$ is generically smooth, we see that the
pull-back morphism $\dK π_X$ of globally defined forms is in fact injective.  It
follows that $\dK f(σ) = 0$ if and only if $(\dK π_X) \bigl( \dK f(σ)\bigr) =
0$.  This last pull-back is easily computed to be zero,
$$
(\dK π_X) \bigl( \dK f(σ)\bigr) = \left(\dK \wtilde f\right) \bigl(\underbrace{\dK π_Y(σ)}_{\mathclap{=0 \text{ by~\eqref{eq:aa}}}}\bigr) = 0.
$$
This ends the proof of Proposition~\ref{prop:pull-back-TFD1}.  \qed

\section{Torsion-free differentials in the relative snc setting}

In the course of the proof of our main theorem, we will frequently need to
consider klt spaces $X$, strong resolution maps $π : \wtilde X → X$, and discuss
torsion-free differentials on the exceptional set $E$, which is an snc divisor
embedded into the smooth space $\wtilde X$.  This section gathers several
elementary facts about the sheaf $\widecheck{Ω}^p_E$ which are needed in the
discussion.  Most of the material here will be known to experts.  We have
nonetheless chosen to include full proofs for lack of an adequate reference.

\subsection{Relatively snc divisors and associated differentials}

\subsubsection{SNC divisors}

To fix notation used later, we recall the definition and basic properties of snc
pairs.

\begin{defn}[\protect{SNC pairs \cite[0.4(8)]{KM98}}]\label{def:everythinglog2a}
  Let $X$ be a normal, quasi-projective variety and $D$ an effective Weil
  divisor on $X$.  Given a point $x ∈ X$, we say that the pair $(X, D)$ is
  \emph{snc at $x$} if there exists a Zariski-open neighbourhood $U$ of $x$ such
  that $U$ is smooth and such that $\supp(D) ∩ U$ is either empty, or a divisor
  with simple normal crossings.  The pair $(X, D)$ is called \emph{snc} if it is
  snc at every point of $X$.

  Given a pair $(X,D)$, let $(X,D)_{\reg}$ be the maximal open set of $X$ where
  $(X,D)$ is snc, and let $(X,D)_{\sing}$ be its complement, with the induced
  reduced subscheme structure.
\end{defn}

The following notation and remark can be used to give an alternate definition of
``snc pair''.  This will be used later to define snc in the relative setting.

\begin{notation}[Intersection of boundary components]\label{not:DI}
  Let $(X,D)$ be a pair, where the boundary divisor $D$ is written as a sum of
  its irreducible components $D = α_1 D_1 + \ldots + α_n D_n$.   If $I
  ⊆ \{1, \ldots, n\}$ is any non-empty subset, we consider the
  scheme-theoretic intersection $D_I := ∩_{i ∈ I} \supp D_i$.  If $I$ is
  empty, set $D_I := X$.
\end{notation}

\begin{rem}[Description of snc pairs]\label{rem:descrSNC}
  In the setup of Notation~\ref{not:DI}, it is clear that the pair $(X,D)$ is
  snc if and only if all $D_I$ are smooth and of codimension equal to the number
  of defining equations, $\codim_XD_I=|I|$ for all $I$ where $D_I \not =
  \emptyset$.
\end{rem}

\subsubsection{SNC morphisms}

The notion of relatively snc divisors has been used in the literature, but its
definition has not been discussed much.  For the reader's convenience, we
reproduce the definition given in \cite[Sect.~2.B]{GKKP11}.

\begin{defn}[\protect{SNC morphism, relatively snc divisor, \cite[Def.~2.1]{VZ02}}]\label{def:sncMorphism}
  If $(X,D)$ is an snc pair and $φ: X → T$ a surjective morphism to a smooth
  variety, we say that $D$ is \emph{relatively snc}, or that $φ$ is \emph{an snc
    morphism of the pair} $(X,D)$ if for any set $I$ with $D_I \not = \emptyset$
  all restricted morphisms $φ|_{D_I} : D_I → T$ are smooth of relative dimension
  $\dim X-\dim T -|I|$.
\end{defn}

\begin{rem}[Geometric description of snc morphisms]\label{rem:sncmor}
  In the setting of Definition~\ref{def:sncMorphism}, assume that $φ: X → T$ is
  an snc morphism of the pair $(X,D)$.  Then, since smooth morphisms are open,
  each of the sets $D_I$ dominates $X$.

  If $t ∈ T$ is any point, set $X_t := φ^{-1}(t)$ and $D_t := D ∩ X_t$.  Then
  $X_t$ is smooth and $(X_t, D_t)$ is an snc pair.
\end{rem}

\begin{example}[SNC morphisms]\label{ex:sncmor}
  \PreprintAndPublication{The morphism $ψ_{α}$, shown in Figure~\vref{fig:TSWAH}
    is a prototypical example of an snc morphism.}{} If $(X,D)$ is an snc pair,
  then the identity map $\Id_X : X → X$ is an snc morphism if and only if $D =
  0$.  Again assuming that $(X,D)$ is an snc pair, the constant map from $X$ to
  a point is always an snc morphism.
\end{example}

Assume we are given an pair $(X, D)$ and a point $x ∈ X$ such that $(X, D)$ is
snc at $x$.  It is well-known that there exists a neighbourhood $U = U(x)$, open
in the analytic topology, and holomorphic coordinates $x_1, \ldots, x_n ∈
\sO_X(U)$ such that $\supp D ∩ U = \{ x_1\cdots x_r = 0 \}$, for a suitable
number $0 ≤ r ≤ n$.  The following is the relative analogue of this fact.

\begin{lemDef}[Adapted coordinates for an snc morphism]\label{lemDef:apdapt}
  Let $(X,D)$ be an snc pair and $φ: X → T$ an snc morphism of the pair $(X,D)$.
  If $x ∈ X$ is any point, then there exist neighbourhoods $V = V\bigl( φ(x)
  \bigr) ⊆ T$ and $U = U(x) ⊆ φ^{-1}(V) ⊆ X$ open in the analytic topology, and
  holomorphic coordinates $x_1, \ldots, x_n ∈ \sO_X(U)$, $y_1, \ldots, y_m ∈
  \sO_Y(V)$, centred about $x$ respectively $φ(x)$, and a number $0 ≤ r ≤ n-m$
  such that the following holds.
  \begin{enumerate}
  \item We have $x_i = y_i ◦ φ$ for all indices $1 ≤ i ≤ m$, and
  \item the support of $D$ is given as $\supp D ∩ U = \{ x_{m+1}\cdots x_{m+r} =
    0 \}$.
  \end{enumerate}
  In the setting above, we call the coordinates $x_{\bullet}$ and $y_{\bullet}$
  \emph{adapted coordinates for the snc morphism $φ$}.  \qed
\end{lemDef}

\subsection{Characterisation of torsion-free differentials}

Given an snc pair $(X,D)$, the following two propositions describe the sheaf of
torsion-free differentials on $D$ by relating them to the sheaves of
differentials on each component of $D$ and to logarithmic differentials on $X$,
respectively.  These descriptions will later be used in the discussion of
relative torsion-free differentials.

\begin{prop}[\protect{Torsion-free differentials on snc divisors, I}]\label{prop:NamiTFD0}
  Let $(X,D)$ be an snc pair where $D$ is reduced, with irreducible components
  $D = ∪_i D_i$.  Let $φ : X → Y$ be a morphism such that $D$ is relatively snc
  over $Y$.  Given any number $p$, consider the inclusion maps $ι_i : D_i → D$
  and the natural morphisms
  $$
  ψ_i: Ω^p_{D/Y} → (ι_i)_* Ω^p_{D_i/Y} \quad \text{and} \quad
  ψ = \oplus ψ_i: Ω^p_{D/Y} → \bigoplus_i (ι_i)_* Ω^p_{D_i/Y}.
  $$
  Then the image of $ψ$ is naturally isomorphic to the sheaf of torsion-free
  differentials on $D$, that is, $\Image(ψ) \cong \widecheck{Ω}^p_{D/Y}$.
\end{prop}
\begin{proof}
  Recall from Proposition~\ref{prop:pftf} that the push-forward sheaves $(ι_i)_*
  Ω^p_{D_i/Y}$ are torsion-free sheaves on the reducible space $D$.  Since
  subsheaves of torsion-free sheaves are torsion-free, Corollary~\ref{cor:sstf},
  it follows that $\Image(ψ)$ is itself torsion-free.  The universal property of
  torsion-freeness, Proposition~\ref{prop:UPTf}, therefore gives a factorisation
  of $ψ$ as follows:
  $$
  \xymatrix{ %
    Ω^p_{D/Y} \ar@{->>}@/^0.4cm/[rrrr]^{ψ} \ar@{->>}[rr] && \widecheck{Ω}^p_{D/Y}
    \ar@{->>}[rr]_(.4){\widecheck{ψ}} && \Image(ψ) \ar@{^(->}[r] & \bigoplus_i (ι_i)_* Ω^p_{D_i/Y}.   }
  $$
  The finish the proof, it suffices to show that $\widecheck{ψ}$ is injective.
  That, however, follows from Proposition~\ref{prop:ijtf} because
  $\widecheck{ψ}$ is generically injective: as a matter of fact, both $ψ$ and
  $\widecheck{ψ}$ are isomorphic away from the singular set of $D$.
\end{proof}

In the absolute case, the following statement appears without proof in
\cite[p.~129]{Namikawa01}.

\begin{prop}[\protect{Torsion-free differentials on snc divisors, II}]\label{prop:NamiTFD}
  In the setup of Proposition~\ref{prop:NamiTFD0}, there exists a short exact
  sequence\footnote{The sheaf $\widecheck{Ω}^p_{D/Y}$ of
    Sequence~\eqref{eq:NamiTFD} should be seen as a sheaf on $X$,
    cf.~Notation~\vref{not:sdos}.}
  \begin{equation}\label{eq:NamiTFD}
    0 \xrightarrow{\qquad} Ω^p_{X/Y}(\log D) \otimes \sJ_D \xrightarrow{\quad A_{X/Y}^p \quad} 
    Ω^p_{X/Y} \xrightarrow{\quad B^p_{X/Y} \quad} \widecheck{Ω}^p_{D/Y} \xrightarrow{\qquad} 0.
  \end{equation}
\end{prop}
\begin{proof}
  Viewing $Ω^p_{D/Y}$ and $Ω^p_{D_i/Y}$ as sheaves on $X$,
  Proposition~\ref{prop:NamiTFD0} asserts that to prove
  Proposition~\ref{prop:NamiTFD}, it suffices to show that $Ω^p_{X/Y}(\log D)
  \otimes \sJ_D$ is exactly the kernel of the composed map
  $$
  \xymatrix{ %
    Ω^p_{X/Y} \ar[r]_r \ar@/^4mm/[rr]^{η} & Ω^p_{D/Y} \ar[r]_(.4){ψ}
    & \bigoplus_i Ω^p_{D_i/Y}.   }
  $$
  Setting $η_i := ψ_i ◦ r : Ω^p_{X/Y} → Ω^p_{D_i/Y}$, the kernel of $η$ is then
  this intersection of the kernels of all the $η_i$, that is,
  \begin{equation}\label{eq:ttz}
    \ker η = \bigcap_i \ker η_i ⊂ Ω^p_{X/Y}.     
  \end{equation}
  The kernels of the $η_i$ are well understood.  As a matter of fact, an
  elementary computation in adapted coordinates shows\footnote{In the absolute
    case, this description of $\ker η_i$ is also found in \cite[Prop.~2.3(c) on
    p.~13]{EV92}.} that
  \begin{equation}\label{eq:ttw}
    \ker η_i = Ω^p_{X/Y}(\log D_i) \otimes \sJ_{D_i}.
  \end{equation}
  In particular, we see that $\ker η_i$ are locally free subsheaves of
  $Ω^p_{X/Y}$.  Since the intersection of two reflexive subsheaves is reflexive,
  Equations~\eqref{eq:ttz} and \eqref{eq:ttw} show that $\ker η$ is a reflexive
  sheaf on $X$, and isomorphic to $Ω^p_{X/Y}(\log D) \otimes \sJ_D$ on the open
  set $X \setminus D_{\sing}$.  Since $\codim_X D_{\sing} ≥ 2$, and since two
  reflexive sheaves agree if and only if they agree on the complement of a small
  set, we obtain that $\ker η \cong Ω^p_{X/Y}(\log D) \otimes \sJ_D$ on all of
  $X$, thus finishing the proof of Proposition~\ref{prop:NamiTFD}.
\end{proof}

Given a smooth morphism $X → Y$, it is well-known that the restriction of
$Ω^p_{X/Y}$ to any fibre $F$ equals $Ω^p_F$.  Using the descriptions given
above, we show that the same holds for torsion-free differentials in the
relative snc setting.

\begin{cor}[Restriction to fibres]\label{cor:ntorrest}
  In the setup of Proposition~\ref{prop:NamiTFD0}, let $y ∈ Y$ be a smooth point
  of the image $φ(D)$, and consider the preimage $D_y := (φ|_D)^{-1}(y)$.  Then
  \begin{equation}\label{eq:recO}
    \widecheck{Ω}^p_{D/Y}\bigl|_{D_y} \cong \widecheck{Ω}^p_{D_y}.
  \end{equation}
\end{cor}
\begin{proof}
  Set $X_y := φ^{-1}(y)$ and $D_y := D ∩ X_y$, and recall from
  Remark~\ref{rem:sncmor} that $(X_y, D_y)$ is an snc pair.  Observe that to
  prove Corollary~\ref{cor:ntorrest}, it suffices to show
  Equation~\eqref{eq:recO} in the neighbourhood of any given point $x ∈ D_y$.
  Given one such $x$, choose open neighbourhoods $U = U(x)$ and $V = V(y)$ and
  adapted coordinates $x_1, \ldots x_n$ and $y_1, \ldots, y_m$ as in Lemma and
  Definition~\ref{lemDef:apdapt}.  Using these coordinates, an elementary
  computation shows that the restriction of Sequence~\eqref{eq:NamiTFD} to the
  fibre $X_y$ stays exact, and that the terms of the restricted sequence are
  identified as follows,
  \begin{equation}\label{eq:restNamSeq}
    0 \longrightarrow \underbrace{\bigl( Ω^p_{X/Y}(\log D) \otimes \sJ_D \bigr)
      \bigl|_{X_y}}_{= Ω^p_{X_y}(\log D_y) \otimes \sJ_{D_y}} \longrightarrow
    \underbrace{Ω^p_{X/Y} \bigl|_{X_y}}_{= Ω^p_{X_y}} \longrightarrow
    \widecheck{Ω}^p_{D/Y} \bigl|_{X_y} \longrightarrow 0.
  \end{equation}
  Recall from Example~\ref{ex:sncmor} that the constant map from $X_y$ to a
  point is an snc morphism of the pair $(X_y, D_y)$.  An application of
  Proposition~\ref{prop:NamiTFD} therefore identifies the cokernel of
  Sequence~\eqref{eq:restNamSeq} as $\widecheck{Ω}^p_{D_y}$, finishing the
  proof.
\end{proof}

\subsection{Filtrations for torsion-free differentials in the relative snc setting}

Given a smooth morphism $X → Y$, the sequence of relative differentials on $X$
induces a canonical filtration on the sheaf $Ω^p_X$.  The following proposition
shows that the same statement holds for torsion-free differentials in the
relative snc setting.

\begin{prop}[Filtration of relative torsion-free differentials for an snc morphism]\label{prop:TFDrelFilt}
  Let $(X,E)$ be an snc pair, where $E$ is a reduced divisor on $X$.  Let $φ : X
  → Y$ a morphism such that $E$ is relatively snc over $Y$.  If $\:0 ≤ p ≤ \dim
  X-1$ is any number, then there exists a filtration
  $$
  \widecheck{Ω}^p_E = \widecheck \sF^0 ⊇ \widecheck \sF^1 ⊇ \cdots ⊇ \widecheck \sF^p ⊇ \widecheck \sF^{p+1} = 0
  $$
  and for all $0 ≤ r ≤ p$ sequences
  \begin{equation}\label{eq:kfgo}
  0 → \widecheck \sF^{r+1} → \widecheck \sF^r → (φ|_E)^* Ω^r_Y \otimes \widecheck{Ω}^{p-r}_{E/Y} → 0.
  \end{equation}
\end{prop}
\begin{proof}
  Since $φ$ is smooth, the sequence of relative differentials on $X$ is a
  short exact sequence of locally free sheaves on $X$,
  $$
  0 → φ^* Ω^1_Y → Ω^1_X → Ω^1_{X/Y} → 0.
  $$
  Following \cite[II.5~Ex~5.16]{Ha77}, there exists an induced filtration
  $$
  Ω^p_X =\ \sF^0 ⊇ \sF^1 ⊇ \cdots ⊇ \sF^p ⊇
  \sF^{p+1} = 0
  $$
  and sequences for all $0 ≤ r ≤ p$,
  $$
  0 → \sF^{r+1} → \sF^r → φ^* Ω^r_Y \otimes Ω^{p-r}_{X/Y} → 0.
  $$
  Recalling Sequence~\eqref{eq:NamiTFD} of Proposition~\ref{prop:NamiTFD}, we
  define filtrations of $Ω^p_{X/Y}(\log E) \otimes \sJ_E$ and
  $\widecheck{Ω}^p_{E/Y}$ by setting
  \begin{align*}
    \sF^r(\log) & := (A_{X/Y}^p)^{-1}\bigl(\sF^r\bigr) && \text{\ldots\ filtration of $Ω^p_{X/Y}(\log E) \otimes \sJ_E$, and}  \\
    \widecheck \sF^r & := B^p_{X/Y}\bigl(\sF^r\bigr)      && \text{\ldots\ filtration of $\widecheck{Ω}^p_{E/Y}$.}
  \end{align*}
  With these definitions, it is clear that $\widecheck \sF^r \cong \sF^r /
  \sF^r(\log)$, for all indices $r$.  We aim to understand the sheaves
  $\widecheck \sF^r$ in more detail.  An explicit computation in adapted
  coordinates, which we leave to the reader, reveals two facts.  First, there is
  an isomorphism
  $$
  \factor \sF^r(\log).\sF^{r+1}(\log).  \,\,\cong\,\, φ^*Ω^r_Y \, \otimes
  \, Ω^{p-r}_{X/Y}(\log E) \otimes \sJ_E.
  $$
  Secondly, one obtains the description of the natural map between quotients
  given in the following commutative diagram
  \begin{equation}\label{eq:Bg1}
    \begin{split}
      \xymatrix{ %
        \factor \sF^r(\log).\sF^{r+1}(\log).  \ar[rrr]^(.55){q}_(.55){\text{natl.  map between quotients}} \ar[d]_{\cong} &&& \factor \sF^r.\sF^{r+1}.  \ar[d]^{\cong} \\
        φ^*Ω^r_Y \otimes \left( Ω^{p-r}_{X/Y}(\log E) \otimes \sJ_E \right) \ar[rrr]^(0.6){{\rm Id}_{φ^*Ω^r_Y} \otimes A^{p-r}_{X/Y}}
        &&& φ^*Ω^r_Y \otimes Ω^{p-r}_{X/Y}.  }    
    \end{split}
  \end{equation}
  In particular, since $A^{p-r}_{X/Y}$ is injective and since $φ^*Ω^r_Y$ is
  locally free on $X$, it follows that $q$ is injective.  The Snake Lemma thus
  yields \PreprintAndPublication{a commutative diagram of coherent sheaves on
    $X$ as follows,
    \begin{equation}\label{eq:Bg2}
      \begin{split}
        \xymatrix{ %
          0 \ar[r] & \sF^{r+1}(\log) \ar[r] \ar@{^{(}->}[d]^{A^p_{X/Y}|_{\sF^{r+1}(\log)}} & \sF^r(\log) \ar[r] \ar@{^{(}->}[d]^{A^p_{X/Y}|_{\sF^r(\log)}} & \factor \sF^r(\log).\sF^{r+1}(\log).  \ar[r] \ar@{^{(}->}[d]^{q} & 0\\
          0 \ar[r] & \sF^{r+1} \ar[r] \ar@{->>}[d] & \sF^r \ar[r] \ar@{->>}[d] & \factor \sF^r.\sF^{r+1}.  \ar[r] \ar@{->>}[d] & 0\\
          0 \ar[r] & \widecheck \sF^{r+1} \ar[r] & \widecheck \sF^r \ar[r] & \factor \widecheck \sF^r.\widecheck \sF^{r+1}.  \ar[r] & 0.\\
        }
      \end{split}
    \end{equation}
    In summary, we obtain the following description of the successive quotients in
    the filtration of $\widecheck{Ω}^p_E$, 
    \begin{align*}
      \factor \widecheck \sF^r.\widecheck \sF^{r+1}.   & \cong \factor \sF^r/\sF^{r+1} .  \sF^r(\log)/\sF^{r+1}(\log).  && \text{Diag.~\eqref{eq:Bg2}}\\
      & \cong \factor φ^*Ω^r_Y \otimes Ω^{p-r}_{X/Y} .  q\left(φ^*Ω^r_Y \otimes \left( Ω^{p-r}_{X/Y}(\log E) \otimes \sJ_E \right)\right).  && \text{Diag.~\eqref{eq:Bg1}}\\
      & \cong φ^*Ω^r_Y \otimes \left( \factor Ω^{p-r}_{X/Y} .  A^{p-r}_{X/Y}\left(Ω^{p-r}_{X/Y}(\log E) \otimes \sJ_E\right) .  \right) && \text{$φ^*Ω^r_Y$ loc.~free}\\
      & \cong  φ^*Ω^r_Y \otimes \widecheck{Ω}^{p-r}_{E/Y} && \text{Seq~\eqref{eq:NamiTFD}}\\
      & \cong  (φ|_E)^*Ω^r_Y \otimes \widecheck{Ω}^{p-r}_{E/Y}.
    \end{align*}}{the following description of the successive quotients in
    the filtration of $\widecheck{Ω}^p_E$, 
    \begin{align*}
      \factor \widecheck \sF^r.\widecheck \sF^{r+1}.   & \cong \factor \sF^r/\sF^{r+1} .  \sF^r(\log)/\sF^{r+1}(\log).  && \text{Snake Lemma}\\
      & \cong \factor φ^*Ω^r_Y \otimes Ω^{p-r}_{X/Y} .  q\left(φ^*Ω^r_Y \otimes \left( Ω^{p-r}_{X/Y}(\log E) \otimes \sJ_E \right)\right).  && \text{Diag.~\eqref{eq:Bg1}}\\
      & \cong φ^*Ω^r_Y \otimes \left( \factor Ω^{p-r}_{X/Y} .  A^{p-r}_{X/Y}\left(Ω^{p-r}_{X/Y}(\log E) \otimes \sJ_E\right) .  \right) && \text{$φ^*Ω^r_Y$ loc.~free}\\
      & \cong  φ^*Ω^r_Y \otimes \widecheck{Ω}^{p-r}_{E/Y} && \text{Seq~\eqref{eq:NamiTFD}}\\
      & \cong  (φ|_E)^*Ω^r_Y \otimes \widecheck{Ω}^{p-r}_{E/Y}.
    \end{align*}}
  This finishes the proof of Proposition~\ref{prop:TFDrelFilt}.
\end{proof}

\section{Torsion-free differentials on rationally chain connected spaces}
\label{sec:tfdrcc}

It is well-known that rationally chain connected compact manifolds do not admit
any differential forms.  Here, we show that the same holds for torsion-free
differentials on rationally chain connected varieties with arbitrary
singularities.

\begin{thm}[Torsion-free differentials on rationally chain connected spaces]\label{thm:diffRCspaces}
  Let $X$ be a reduced, projective scheme.  Assume that $X$ is rationally chain
  connected.  Then
  $$
  H^0 \bigl( X,\, \widecheck{Ω}^p_X \bigr) = 0, \quad \text{for all $0 < p ≤
    \dim X$.}
  $$
\end{thm}

\begin{rem}
  In Theorem~\ref{thm:diffRCspaces}, we do not assume that $X$ is irreducible.
\end{rem}

\begin{warning}[Kähler differentials on rationally chain connected spaces]
  The statement of Theorem~\ref{thm:diffRCspaces} becomes wrong if one replaces
  torsion-free differentials with Kähler differentials.  For an example, let $X
  = X_1 ∪ X_2$ be the union of two distinct lines in $\bP^2$.  The reducible
  variety $X$ is then clearly rationally chain connected.  The sheaf $Ω^1_X$ of
  Kähler differentials contains a non-trivial torsion subsheaf, supported at the
  intersection point $X_1 ∩ X_2$.  As a skyscraper sheaf supported at a single
  point, $\tor Ω^1_X$, and hence $Ω^1_X$, do have non-trivial sections.

  Similar examples exist where $X$ is irreducible and rationally connected.  The
  paper \cite{GrebRollenske} discusses cones where $\tor Ω^1_X$ is non-trivial,
  and supported at the vertex point.
\end{warning}

\begin{warning}[Reflexive differentials on rationally chain connected spaces]
  The statement of Theorem~\ref{thm:diffRCspaces} becomes wrong if one replaces
  torsion-free differentials with reflexive differentials;
  Example~\vref{ex:nlc2} discusses a non-trivial reflexive form on the
  (rationally chain connected) cone over an elliptic curve.

  The main reason for failure Theorem~\ref{thm:diffRCspaces} in this setting is
  the fact that $X$ has log canonical rather than klt singularities.  For
  rationally chain connected spaces $X$ with klt singularities, it is shown in
  \cite[Thm.~5.1]{GKKP11} that $H^0 \bigl( X,\, Ω^{[p]}_X \bigr) = 0$ for all
  $p$.  For reflexive tensor operation other than $\bigwedge^{[p]}$, the
  question becomes rather subtle.  We refer to \cite[Sect.~3]{GKP11} for a
  discussion of known facts and for examples.
\end{warning}

The arguments used to prove of Theorem~\ref{thm:diffRCspaces} follow the
standard proof for the non-existence of Kähler differentials on rationally
connected manifolds.

\subsection{Proof of Theorem~\ref*{thm:diffRCspaces}}

Let $σ ∈ H^0 \bigl( X,\, \widecheck{Ω}^p_X \bigr)$ be any reflexive
differential.  By Remark~\ref{gen:charTFD}, it suffices to show that $σ$
vanishes at the general point of every irreducible component of $X$.  In other
words, given any irreducible component $X_0 ⊆ X$ with inclusion map $ι : X_0 →
X$, we need to show that $\dnoTor ι (σ) = 0$.

The assumption that $X$ is rationally chain connected implies that there exists
a fixed point $y ∈ X$, such that general points of $X_0$ can be connected to $y$
using a finite-length chain of rational curves.  The following is thus an
immediate consequence.

\begin{consequence}\label{cons:RCC}
  There exists a nodal chain of smooth rational curves $C = C_1 ∪ \cdots ∪ C_k$
  with marked points $p_i$ as as shown in Figure~\ref{fig:chain}, a smooth,
  irreducible, quasi-projective variety $H ⊂ \Hom(C,X)$ and a morphism $μ : H ⨯
  C → X$ with the following properties.
  \begin{enumerate}
  \item\label{il:bonny} The variety $H ⨯ C_1$ dominates the component $X_0$.
  \item\label{il:clyde} The restriction of $μ$ to $H ⨯ \{ p_k\}$ is constant.
    \qed
  \end{enumerate}
\end{consequence}

\begin{figure}
  \centering
  \begin{tikzpicture}
    \draw (0,0) to [out=35, in=145] (2.5,0);
    \fill (1.25,0.7) node {$C_1$};
    \draw (2,0) to [out=35, in=145] (4.5,0);
    \fill (3.25,0.7) node {$C_2$};
    \draw (6,0) to [out=35, in=145] (8.5,0);
    \fill (7.25,0.7) node {$C_{k-1}$};
    \draw (8,0) to [out=35, in=145] (10.5,0);
    \fill (9.25,0.7) node {$C_k$};

    \fill (2.25,0.15) circle (2pt);
    \fill (2.25,0) node[below] {$p_1$};
    \fill (4.25,0.15) circle (2pt);
    \fill (4.25,0) node[below] {$p_2$};
    \fill (6.25,0.15) circle (2pt);
    \fill (6.25,0) node[below] {$p_{k-2}$};
    \fill (8.25,0.15) circle (2pt);
    \fill (8.25,0) node[below] {$p_{k-1}$};
    \fill (10.25,0.15) circle (2pt);
    \fill (10.25,0) node[below] {$p_k$};
    \fill (5.25,0.3) node[below] {\ldots};
  \end{tikzpicture}
  
  \caption{A pointed chain of smooth rational curves, as used in the proof of Theorem~\ref{thm:diffRCspaces}}
  \label{fig:chain}
\end{figure}
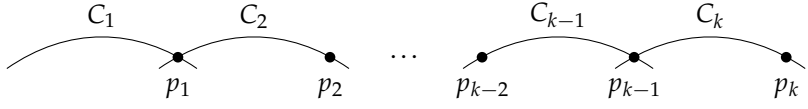

With the notation introduced above, the following Lemma is key to the proof of
Theorem~\ref{thm:diffRCspaces}.

\begin{lem}\label{lem:XP}
  Setup as above.  For any index $1 ≤ j ≤ k$, we consider the following two
  restrictions of $μ$,
  $$
  μ_{p_j} : H ⨯ \{p_j\} → X \quad \text{and} \quad μ_{C_j} : H ⨯ C_j → X.
  $$
  Then the following two equations hold for all indices $1 ≤ j ≤ k$
  \begin{align}
    \label{eq:castor} \dnoTor  μ_{p_j}(σ) & = 0 ∈ H^0 \bigl( H ⨯ \{p_j\},\, Ω^p_{H ⨯ \{p_j\}} \bigr) && \text{and}\\
    \label{eq:pollux} \dnoTor  μ_{C_j}(σ) & = 0 ∈ H^0 \bigl( H ⨯ C_j,\, Ω^p_{H ⨯ C_j} \bigr)
  \end{align}
\end{lem}

Assuming for a second that Lemma~\ref{lem:XP} holds, consider
Equation~\eqref{eq:pollux} for $j=1$.  Write $μ_{C_1}: H ⨯ C_1$ as a composition
$$
\xymatrix{ %
   H ⨯ C_1 \ar[rr]^{μ_{C_1,0}} && X_0 \ar[r]^{ι} & X
}
$$
Equation~\eqref{eq:pollux} and the composition law for pull-back of torsion-free
differentials, Lemma~\ref{lem:clTFD} give a vanishing of forms,
$$
0 = \dnoTor μ_{C_1} (σ) = \dnoTor μ_{C_1,0} \bigl( \dnoTor ι (σ) \bigr)
$$
Recalling from (\ref{cons:RCC}.\ref{il:bonny}) that $μ_{C_1,0}$ is dominant,
Equation~\eqref{eq:pollux} says that the restriction of $σ$ to the component
$X_0$ vanishes generically.  The required vanishing of $\dnoTor ι (σ)$ then an
immediate consequence.  Theorem~\ref{thm:diffRCspaces} is thus a consequence of
Lemma~\ref{lem:XP}.

\subsubsection{Proof of Lemma~\ref*{lem:XP}}

It remains to prove Lemma~\ref{lem:XP}.  The proof proceeds by descending
induction on $j$.  To this end, the following statements will be shown in
Section~\ref{sssec:s1}--\ref{sssec:s3} below.

\begin{claim}[Start of induction]\label{claim:XP1}
  Equation~\eqref{eq:castor} holds for $j=k$.
\end{claim}

\begin{claim}[Inductive step I]\label{claim:XP2}
  For all indices $1 ≤ j ≤ k$, Equation~\eqref{eq:pollux} holds for the
  index $j$ if Equation~\eqref{eq:castor} holds for $j$.
\end{claim}

\begin{claim}[Inductive step II]\label{claim:XP3}
  For all indices $1 < j ≤ k$, Equation~\eqref{eq:castor} holds for the index
  $j-1$ if Equation~\eqref{eq:pollux} holds for $j$.
\end{claim}

\subsubsection{Proof of Claim~\ref*{claim:XP1}}
\label{sssec:s1}

For $j=k$, Equation~\eqref{eq:castor} follows from
(\ref{cons:RCC}.\ref{il:clyde}), which asserts that $μ_{p_k}$ is constant.  The
pull-back map of Kähler differentials is thus zero, $d μ_{p_k} = 0$.  By
Lemma~\ref{lem:dandcheckd}, the pull-back map of torsion-free differentials is
therefore zero too, so that $\dnoTor μ_{p_k} = 0$ as claimed.  \qed

\subsubsection{Proof of Claim~\ref*{claim:XP2}}
\label{sssec:s2}

Let $1 ≤ j ≤ k$ be any given index, and assume that Equation~\eqref{eq:castor}
holds for $j$.  The following morphisms are relevant in our discussion
$$
\xymatrix{ %
H ⨯ \{ p_j \} \ar[rr]_{γ, \text{ inclusion}} \ar@/^0.5cm/[rrrr]^{μ_{p_j}} &&  H ⨯ C_j \ar[d]^{π_H,\text{ projection}} \ar[rr]_{μ_{C_j}} && X\\
&& H
}
$$
The product structure of $H⨯ C_j$ immediately gives a splitting
$$
Ω^p_{H ⨯ C_j} \,\, \cong \,\, \underbrace{π_H^* (Ω^p_H)}_{=: \sA} \,\, \oplus
\,\, \underbrace{π_H^* (Ω^{p-1}_H) \otimes π_{C_j}^* (Ω^1_{C_j})}_{=: \sB}
$$
Decompose $\dnoTor μ_{C_j}(σ) ∈ H^0\bigl(H ⨯ C_j ,\, Ω^p_{H ⨯ C_j} \bigr)$
accordingly as $\dnoTor μ_{C_j}(σ) = \widecheck{σ}_{\sA} + \widecheck{σ}_{\sB}$,
we aim to show that both $\widecheck{σ}_{\sA}$ and $\widecheck{σ}_{\sB}$ are
zero.

First, if $F = \{h\}⨯ C_j \cong \bP^1$ is any fibre of $π_H$, then $\sB|_F \cong
\sO_{\bP^1}(-2)^{\oplus \bullet}$ is anti-ample.  It follows that $\sB$ only has
the trivial section, so $\widecheck{σ}_{\sB} = 0$.  Secondly, if follows from
Equation~\eqref{eq:castor} and from the composition law for pull-back of
torsion-free differentials, Lemma~\ref{lem:clTFD}, that
$$
\dnoTor γ \bigl( \dnoTor μ_{C_j} (σ) \bigr) = \widecheck{μ}_{p_j}(σ) = 0.
$$
We obtain that $\widecheck{σ}_{\sA}|_{H ⨯ \{p_j\}} = 0$.  Since $\sA|_F$ is a
trivial vector bundle, $\sA|_F \cong \sO_{\bP^1}^{\oplus \bullet}$, this shows
vanishing of $\widecheck{σ}_{\sA}$.

In summary, we have seen that $\dnoTor μ_{C_j}(σ)$ is zero, as asserted in
Claim~\ref*{claim:XP2}.  \qed

\subsubsection{Proof of Claim~\ref*{claim:XP3}}
\label{sssec:s3}

Let $1 < j ≤ k$ be any given index, and assume that Equation~\eqref{eq:pollux}
holds for the index $j$.  Since $p_{j-1} ∈ C_j$, we obtain a factorisation
$$
\xymatrix{ %
  H ⨯ \{ p_{j-1} \} \ar[rr]_{γ,\text{ inclusion}}
  \ar@/^0.5cm/[rrrr]^{μ_{p_{j-1}}} && H ⨯ C_j \ar[rr]_{μ_{C_j}} && X }
$$
The composition law for the pull-back of torsion-free differentials,
Lemma~\ref{lem:clTFD}, then asserts that
$$
\dnoTor μ_{p_{j-1}} (σ) = \dnoTor γ \bigl( \dnoTor μ_{C_j}(σ)
\bigr) \underset{\eqref{eq:pollux}}{=} \dnoTor γ(0) = 0,
$$
which immediately shows the desired vanishing.  This finishes the proof of
Claim~\ref{claim:XP3}, and hence of Theorem~\ref{thm:diffRCspaces}.  \qed

\part{Pull-back properties of reflexive differentials on klt spaces}
\label{part:2}

\section{Main result and elementary consequences}
\label{sec:5}

\subsection{Formulation of  the main result}

We aim to construct a pull-back map for reflexive differentials, which will turn
out to be uniquely determined by universal properties.  To formulate these
properties (``composition law'', ``compatibility with Kähler differentials'') in
a technically correct manner, it seems easiest to use the language of functors.
The following definition fixes the category.

\begin{defn}[Category of klt base spaces]
  Let $X$ be a normal, irreducible variety.  We call $X$ a \emph{klt base space}
  if there exists a $\bQ$-divisor $D$ on $X$ such that the pair $(X,D)$ is klt.
  A morphism between klt base spaces is simply a morphism of varieties.
\end{defn}

The following theorem contains the main result of this paper.  Its proof is given
in Sections~\ref{ssec:pfcpb}--\ref{sec:pfPB}, starting on
Page~\pageref{ssec:pfcpb}.  An extension of Theorem~\ref{thm:PB-thmA} to
morphisms with arbitrary domain is discussed in Section~\vref{ssec:mwad}.

\begin{thm}[Pull-back map for reflexive differentials on klt base spaces]\label{thm:PB-thmA}
  There exists a unique contravariant functor
  \begin{equation}\label{eq:CVF}
    \begin{matrix}
      \drefl : & \{ \text{klt base spaces}  \} & → & \{ \bC\text{-vector spaces} \}\\
      & X & \mapsto & H^0 \bigl( X,\, Ω^{[p]}_X \bigr)
    \end{matrix}
  \end{equation}
  that satisfies the following ``Compatibility with Kähler differentials''.  If
  $f : Z → X$ is a morphism of klt base spaces such that the open set $Z° :=
  Z_{\reg} ∩ f^{-1}(X_{\reg})$ is not empty, then there exists a commutative
  diagram,
  \begin{equation}\label{eq:BCD}
    \begin{split}
      \xymatrix{ %
        H^0 \bigl( X,\, Ω^{[p]}_X \bigr) \ar[rr]^{\drefl f} \ar[d]_{\restr_X} && H^0 \bigl( Z,\, Ω^{[p]}_Z \ar[d]^{\restr_Z} \bigr) \\
        H^0 \bigl( X_{\reg},\, Ω^p_{X_{\reg}} \bigr) \ar[rr]_{\dK(f|_{Z°})} && H^0 \bigl( Z°,\, Ω^p_{Z°} \bigr), }
    \end{split}
  \end{equation}
  where $\dK (f|_{Z°})$ denotes the usual pull-back of Kähler
  differentials, and $\drefl f$ denotes the linear map of complex vector spaces
  induced by the contravariant functor \eqref{eq:CVF}.
\end{thm}

\begin{rem}[Restrictions used in Diagram~\eqref{eq:BCD}]\label{rem:43}
  In the setting discussed in Diagram~\eqref{eq:BCD}, we have equalities
  $$
  Ω^{[p]}_X\bigl|_{X_{\reg}} = Ω^p_{X_{\reg}} \quad \text{and} \quad
  Ω^{[p]}_Z\bigl|_{Z°} = Ω^p_{Z°}.
  $$
  This justifies the use of the word ``restriction'' in Diagram~\eqref{eq:BCD}.
  Since $X$ and $Z$ are normal, hence smooth in codimension one, the restriction
  maps $\restr_X$ and $\restr_Z$ are clearly isomorphic.
\end{rem}

\begin{rem}[Compatibility with Kähler differentials in special cases]
  Consider a morphism $f : Z → X$ of klt base spaces whose image is contained in
  the singular locus of $X$.  In this case, the set $Z°$ discussed in
  Theorem~\ref{thm:PB-thmA} is empty.  The compatibility condition formulated in
  the theorem is then also empty, that is, always satisfied.  This does not mean
  that $\drefl f$ is an arbitrary map.  The pull-back map $\drefl f$ is in fact
  uniquely defined by the functorial properties (``composition rule''), and by
  the requirement that the pull-back maps of other, dominant, morphisms need to
  satisfy compatibility with Kähler differentials.  The proof of
  Proposition~\ref{prop:cmm} will illustrate this principle.
\end{rem}

\begin{rem}[Sanity of notation]
  Given a klt base space, the functorial formulation of
  Theorem~\ref{thm:PB-thmA} would in principle allow to write $\drefl X$ as a
  shorthand for the space $H^0 \bigl( X,\, Ω^{[p]}_X \bigr)$ of reflexive
  differential forms on $X$.  In order to avoid confusion and incompatibility
  with the literature, we will never use this notation.
\end{rem}

\subsection{Elementary properties of the pull-back map}
\label{ssec:5B}

The compatibility with Kähler differentials implies that practically all
properties known from the pull-back map of Kähler differentials also hold in the
reflexive setting.  We mention some of the more immediate examples in the
present Section~\ref{ssec:5B}.

\begin{prop}[Compatibility with open immersions]\label{prop:coi}
  Let $X$ be any klt base space, $Z ⊆ X$ any open set and $f : Z → X$ the
  inclusion map.  Then
  $$
  \drefl f : H^0 \bigl( X,\, Ω^{[p]}_X \bigr) → H^0 \bigl( Z,\,
  Ω^{[p]}_Z \bigr)
  $$
  equals the standard restriction map.  \PreprintAndPublication{}{\qed}
\end{prop}
\PreprintAndPublication{
  \begin{proof}
    Define the subset $Z°$ as in Theorem~\ref{thm:PB-thmA} and observe that $Z°$
    is not empty.  Using the isomorphisms $\restr_X$ and $\restr_Y$, the claim
    follows from the observation that $\dK(f|_{Z°})$ is the standard restriction
    map.
  \end{proof}
}{}

\begin{prop}[Morphisms to smooth target spaces]\label{prop:mtsmts}
  Let $f : Z → X$ be any dominant morphism between klt base spaces, where $X$ is
  smooth.  Then $\drefl f$ equals the composition of the following maps
  $$
  H^0 \bigl( X,\, Ω^{[p]}_X \bigr) = H^0 \bigl( X,\, Ω^p_X \bigr)
  \xrightarrow{\quad\dK\quad} H^0 \bigl( Z,\, Ω^p_Z \bigr)
  \xrightarrow{\quad\varphi\quad} H^0 \bigl( Z,\, Ω^{[p]}_Z \bigr),
  $$
  where $\varphi$ is induced by the standard map from the sheaf $Ω^p_Z$ into its
  double dual.  \qed
\end{prop}

\begin{prop}[Pull-back via a dominant morphism]\label{prop:pbdm}
  Let $f : Z → X$ be any dominant morphism between klt base spaces.  Then the
  associated map
  $$
  \drefl f : H^0 \bigl( X,\, Ω^{[p]}_X \bigr) → H^0 \bigl( Z,\, Ω^{[p]}_Z \bigr)
  $$
  is injective.  \PreprintAndPublication{}{\qed}
\end{prop}
\PreprintAndPublication{
  \begin{proof}
    Define the subset $Z°$ as in Theorem~\ref{thm:PB-thmA} and observe that $Z°$
    is not empty, and that $f|_{Z°} : Z° → X$ is dominant.
    Proposition~\ref{prop:pbdm} then follows immediately from injectivity of
    $\dK(f|_{Z°})$.
  \end{proof}
}{}

\subsubsection{The pull-back map on sheaf level}

The following Proposition~\ref{prop:cmm} and Corollary~\ref{cor:pbsheaf} imply
that the pull-back map for reflexive differentials is already defined on sheaf
level, just as the pull-back map for Kähler differentials is.

\begin{prop}[Compatibility with module multiplication]\label{prop:cmm}
  Let $f : Z → X$ be any morphism between klt base spaces.  Then associated map
  $$
  \drefl f : H^0 \bigl( X,\, Ω^{[p]}_X \bigr) → H^0 \bigl( Z,\,
  Ω^{[p]}_Z \bigr)
  $$
  is a morphism of $\sO_X(X)$-modules.
\end{prop}

\begin{cor}[Pull-back map on sheaf level]\label{cor:pbsheaf}
  If $f : Z → X$ is any morphism between klt base spaces, then the pull-back
  morphism
  $$
  \drefl f : H^0 \bigl( X,\, Ω^{[p]}_X \bigr) → H^0 \bigl( Z,\, Ω^{[p]}_Z \bigr)
  $$
  is induced by a sheaf morphism
  $$
  \drefl f : f^* Ω^{[p]}_X → Ω^{[p]}_Z.
  $$
\end{cor}
\begin{proof}[Proof of Corollary~\ref{cor:pbsheaf}]
  Immediate from compatibility with open immersions and compatibility with
  module multiplication, Propositions~\ref{prop:coi} and \ref{prop:cmm}.
\end{proof}

The proof of Proposition~\ref{prop:cmm} makes use of the following elementary
construction, which we note for later reference.

\begin{construction}\label{cons:exV}
  Let $f : Z → X$ be any morphism between normal varieties, and let $π: \wtilde
  X → X$ be a resolution of singularities.  Then there exists a smooth variety
  $V$ and a commutative diagram,
  \begin{equation}\label{eq:cSD}
    \begin{split}
      \xymatrix{ %
        &&&& \wtilde X \ar[d]_{π}^{\txt{\scriptsize resolution of\\\scriptsize singularities}} \\
        V \ar@/^5mm/[rrrru]^{a} \ar[rr]_{g\text{, surjective}} && Z \ar[rr]_f && X.
      }
    \end{split}
  \end{equation}
  One way of construction goes as follows.  Choose a component $Y ⊆ \wtilde X
  ⨯_X Z$ that surjects onto $Z$, and let $\wtilde Y → Y$ be a desingularisation.
  Let $d$ be the relative dimension of $\wtilde Y → Z$, and $V ⊂ \wtilde Y$ the
  intersection of $d$ general hyperplanes (if $d=0$, set $V := \wtilde Y$).  The
  variety $V$ is then smooth, with natural morphisms to $Z$ and $\wtilde X$
  making the diagram commutative.  The morphism to $g: V → Z$ constructed in
  this manner is surjective and generically finite.
\end{construction}

\begin{proof}[Proof of Proposition~\ref{prop:cmm}]
  Given reflexive forms $σ_1, σ_2 ∈ H^0 \bigl( X,\, Ω^{[p]}_X \bigr)$ and a
  function $τ ∈ H^0 \bigl( X,\, Ω^{[p]}_X \bigr)$, we need to show that
  \CounterStep
  \begin{equation}\label{eq:cmm1}
    (\drefl f) (τ \cdot σ_1 + σ_2) \overset{!}{=} τ \cdot (\drefl f)
    (σ_1) + (\drefl f)(σ_2).
  \end{equation}
  There are two situation where this is easily true.
  \begin{enumerate}
    \setcounter{enumi}{\value{equation}}
  \item\label{il:XA} If the target of $f$ is smooth, then
    Proposition~\ref{prop:mtsmts} implies that~\eqref{eq:cmm1} holds because
    both $\dK$ and $\varphi$ are $\sO_X(X)$-linear.
  \item\label{il:XB} If $f$ is dominant, then $Z°$ is non-empty, and
    Equation~\eqref{eq:cmm1} holds because compatibility with Kähler
    differentials implies that it holds on the open set $Z°$.
    \setcounter{equation}{\value{enumi}}
  \end{enumerate}
  Proposition~\ref{prop:cmm} is hence shown for dominant morphisms, and for
  morphisms with smooth target.  If $f$ is neither, apply
  Construction~\ref{cons:exV} to obtain a commutative diagram as
  in~\eqref{eq:cSD}.  We have seen above that $\drefl g$ is $\sO_Z(Z)$-linear.
  It is injective by Proposition~\ref{prop:pbdm}.  To prove~\eqref{eq:cmm1}, it
  will therefore suffice to prove the analogous equation for the composed
  morphism $f◦ g$,
  $$
  \bigl( \drefl (f ◦ g) \bigr) (τ \cdot σ_1 + σ_2) \overset{!}{=} τ \cdot \bigl( \drefl (f ◦ g) \bigr)
  (σ_1) + \bigl( \drefl (f ◦ g) \bigr) (σ_2).
  $$
  That, however, follows from functoriality, $\drefl (f ◦ g) = \bigl( \drefl π
  \bigr) ◦ \bigl( \drefl a \bigr)$, because we have seen in above that $\drefl
  π$ and $\drefl a$ are linear over global sections of their respective
  structure sheaves.
\end{proof}

\subsubsection{Compatibility with wedge products and exterior derivatives}
\label{ssec:5C}

If $X$ is any normal variety, then to give a reflexive differential on $X$ it is
equivalent to give a Kähler differential on the smooth locus $X_{\reg}$.  More
precisely, if $ι: X_{\reg} → X$ denotes the inclusion of the smooth locus into
$X$, then $Ω^{[p]}_X = ι_* Ω^p_{X_{\reg}}$.  This description of reflexive
differentials immediately allows to define \emph{reflexive wedge products}, that
is, for all $p$ and $q$ morphisms
$$
\begin{array}{rccc}
  \wedge : & Ω^{[p]}_X \otimes_{\sO_X} Ω^{[q]}_X & → & Ω^{[p+q]}_X \\
  \wedge : & H^0 \bigl( X,\, Ω^{[p]}_X \bigr) \otimes_{\sO_X(X)} H^0 \bigl( X,\, Ω^{[q]}_X \bigr) & → & H^0 \bigl( X,\, Ω^{[p+q]}_X \bigr)
\end{array}
$$
that, on the smooth locus of $X$, agree with the usual wedge products.  We
define $\bC$-linear \emph{reflexive exterior derivatives} in the same fashion,
for all $p$,
$$
\begin{array}{rccc}
  d : & Ω^{[p]}_X & → & Ω^{[p+1]}_X \\
  d : & H^0 \bigl( X,\, Ω^{[p]}_X \bigr) & → & H^0 \bigl( X,\, Ω^{[p+1]}_X \bigr).
\end{array}
$$

In this setting, the proof of Proposition~\ref{prop:cmm} applies almost verbatim
to give the following compatibility result.

\begin{prop}[Compatibility with wedge products and exterior derivatives]
  Let $f : Z → X$ be any morphism between normal varieties.  Then the pull-back
  maps for reflexive differentials and sheaves of reflexive differentials
  commute with reflexive wedge products and reflexive exterior derivatives.
  \qed
\end{prop}

\subsection{Morphisms with arbitrary domain}
\label{ssec:mwad}

The properties listed above allow for the construction of a meaningful pull-back
map even for morphisms where only the target is assumed to be a klt base space.

To this end, let $f : Z → X$ be any morphism between normal varieties and assume
that $X$ a klt base space.  Observing that the smooth locus $Z_{\reg}$ is a klt
base space as well, and recalling that the restriction map
$$
\restr_Z : H^0 \bigl( Z,\, Ω^{[p]}_Z \bigr) → H^0 \bigl( Z_{\reg},\,
Ω^p_{Z_{\reg}} \bigr)
$$
is isomorphic, define a pull-back map for reflexive differentials as the
composition of the following two maps,
$$
H^0 \bigl( X,\, Ω^{[p]}_X \bigr) \xrightarrow{\quad\drefl (f|_{X_{\reg}})
  \quad} H^0 \bigl( Z_{\reg},\, Ω^p_{Z_{\reg}} \bigr)
\xrightarrow{\quad\restr_Z^{-1}\quad} H^0 \bigl( Z,\, Ω^{[p]}_Z \bigr).
$$
As there is no possibility of confusion, we denote this map again by $\drefl f$.
Is is easy to show that all properties listed in Section~\ref{ssec:5B} also hold
for this generalised map.  In particular, $\drefl f$ is induced by a sheaf
morphism.  We have thus constructed maps $ \drefl f : Ω^{[p]}_X → Ω^{[p]}_Z$ and
$\drefl f : H^0 \bigl( X,\, Ω^{[p]}_X \bigr) → H^0 \bigl( Z,\, Ω^{[p]}_Z \bigr)$
for all indices $p$.

\section{Preparation for the proof of Theorem~\ref*{thm:PB-thmA}}
\label{ssec:pfcpb}

Given a morphism of klt base spaces and a reflexive form on the target space,
the following proposition constructs a reflexive form on the domain that
satisfies a weak universal property.

\begin{prop}[Construction of pull-back forms]\label{prop:const}
  Let $f: Z → X$ be any morphism of klt base spaces, and let $σ ∈ H^0\bigl( X,\,
  Ω^{[p]}_X \bigr)$ be any reflexive $p$-form on $X$.  Then there exists a
  unique reflexive $p$-form $τ ∈ H^0\bigl( Z,\, Ω^{[p]}_Z \bigr)$ satisfying the
  following universal property.  Given any commutative diagram
  \begin{equation}\label{eq:const}
    \begin{split}
      \xymatrix{ %
        &&&& \wtilde X \ar[d]_{π}^{\txt{\scriptsize resolution of\\\scriptsize singularities}} \\
        V \ar@/^5mm/[rrrru]^{a} \ar[rr]_{g\text{, surjective}} && Z_{\reg} \ar[rr]_{f|_{Z_{\reg}}} && X,
      }
    \end{split}
  \end{equation}
  where $V$ is smooth, let $\wtilde{σ} ∈ H^0\bigl( \wtilde X,\, Ω^p_{\wtilde X}
  \bigr)$ be the unique differential form on $\wtilde X$ that agrees with $σ$
  wherever $π$ is isomorphic.  Then
  \begin{equation}\label{eq:const2}
    \dK a(\wtilde{σ}) = 
    \dK g(τ|_{Z_{\reg}}) ∈ H^0 \bigl( V,\, Ω^p_V \bigr).
  \end{equation}
\end{prop}

\begin{rem}[Existence of $\wtilde{σ}$, restriction of reflexive differentials]
  The existence of the form $\wtilde{σ}$ used in the formulation of
  Proposition~\ref{prop:const} is the main result of the paper \cite{GKKP11}.
  We refer to \cite[Thm.~1.4 and Rem.~1.5.2]{GKKP11} for details.  In analogy
  with Remark~\ref{rem:43}, we have used the canonical identification
  $Ω^{[p]}_Z|_{Z_{\reg}} \cong Ω^p_{Z_{\reg}}$ implicitly in the formulation of
  Equation~\eqref{eq:const2}.
\end{rem}

\begin{rem}[The case where $f(Z) \not \subset X_{\sing}$]\label{rem:exV2}
  In the setting of Proposition~\ref{prop:const}, if $f(Z)$ is not contained in
  the singular locus of $X$, then it follows immediately from the universal
  property \eqref{eq:const2} that $τ$ is the unique reflexive form whose
  restriction to the open set $Z° := Z_{\reg} ∩ f^{-1}(X_{\reg})$ satisfies
  $τ|_{Z°} = \dK(f|_{Z°}) (σ|_{X_{\reg}})$.
\end{rem}

\begin{rem}[Existence of $V$, $a$ and $g$ for given resolution $π$]\label{rem:exV}
  Given an arbitrary resolution of singularities $π: \wtilde X → X$,
  Construction~\ref{cons:exV} shows that there always exist smooth varieties $V$
  and morphisms $a$, $g$ as in Diagram~\eqref{eq:const}.
\end{rem}

\subsection{Preparation for the proof of Proposition~\ref*{prop:const}}
\label{ssec:preppf}

We will see in Section~\ref{ssec:popropconst} that Proposition~\ref{prop:const}
is in fact a corollary of the following, seemingly weaker lemma.

\begin{lem}[Weak version of Proposition~\ref{prop:const}]\label{lem:const}
  Let $f: Z → X$ be any morphism of klt base spaces, and let $σ ∈ H^0\bigl( X,\,
  Ω^{[p]}_X \bigr)$ be any reflexive $p$-form on $X$.  Then there exists a
  resolution of singularities $π: \wtilde X → X$ and a reflexive $p$-form $τ ∈
  H^0\bigl( Z,\, Ω^{[p]}_Z \bigr)$ satisfying a universal property similar to
  the one spelled out in Proposition~\ref{prop:const}: given any smooth variety
  $V$ and morphisms $a$, $g$ forming a commutative diagram as
  in~\eqref{eq:const}, then Equation~\eqref{eq:const2} holds.
\end{lem}

\begin{rem}[Relation between Lemma~\ref{lem:const} and Proposition~\ref{prop:const}]
  We will later see that the form $τ$ constructed in Lemma~\ref{lem:const} for
  one specific resolution map will also work for any other.
\end{rem}

The remainder of the present Section~\ref{ssec:preppf} is devoted to the proof
of Lemma~\ref{lem:const}.  Though not extremely involved on a conceptual level,
the proof is somewhat lengthy to write down.  To help the reader maintain an
overview, we have subdivided the proof into a number of relatively independent
steps, given in Sections~\ref{sssec:PBAnotation}--\ref{sssec:polconstend} below.

\subsubsection{Proof of Lemma~\ref*{lem:const}: setup of notation, simplification}
\label{sssec:PBAnotation}

We maintain assumptions and notation of Lemma~\ref{lem:const} throughout the
proof.  In order to construct a reflexive form on $Z$, it suffices to construct
the form away from a set of codimension two.  Since $Z$ is a klt base space,
hence normal, we can therefore assume without loss of generality that the
following holds.

\begin{aassumption}\label{aass:smooth}
  The space $Z$ is smooth.
\end{aassumption}

In case where the image of $f$ is not contained in the singular set of $X$, it
has been shown in \cite[Thm.~4.3]{GKKP11} that there exists a differential form
$τ$ on $Z$ that agrees on the open set $f^{-1}(X_{\reg})$ with the usual
pull-back of the Kähler differential $σ|_{X_{\reg}}$.  The differential $τ$
clearly satisfies all requirements stated in Lemma~\ref{lem:const}, so that the
proof is already finished in this case.  We will therefore assume without loss
of generality that the following holds.

\begin{aassumption}\label{aass:sing}
  The image of $f$ is contained in the singular set of $X$.
\end{aassumption}

We define $T ⊆ X$ as the Zariski closure of the image of $f$, that is, $T :=
\overline{f(Z)}$.  Choose a desingularisation $π : \wtilde X → X$ with the
additional property that the preimage $π^{-1}(T) ⊂ \wtilde X$ has pure
codimension one and forms a divisor with simple normal crossing support.
Finally, let $E ⊂ π^{-1}(T)$ be the union of those components that dominate
(=surject onto) the irreducible variety $T$.  Its irreducible components are
denoted as $E = E_0 ∪ \cdots ∪ E_k$

\subsubsection{Proof of Lemma~\ref*{lem:const}: projection to general points of $T$}
\label{sssec:proj}

One way to describe the geometry of $X$ near general points of $T$ is by looking
at a family of sufficiently general complete intersections $(H_t)_{t ∈ T}$, and
by studying the varieties $H_t$ at their intersection points with $T$.  At
general points of $T$, the family defines a morphism, and it is often
notationally convenient to discuss the varieties $H_t$ as being fibres of that
morphism.  This idea is not new, and is explained in great detail in
\cite[Sect.~2.G]{GKKP11}.  The following lemma summarises the results and fixes
notation used throughout the remainder of the proof of Lemma~\ref{lem:const}.
\PreprintAndPublication{Figure~\ref{fig:TSWAH}, taken from the preprint version
  of the paper \cite{GKKP11}, illustrates the setup.}{}

\PreprintAndPublication{
\begin{figure}
  \centering

  \unitlength 1cm
  $$
  \xymatrix{
    *++{\begin{picture}(4,4)(0,0)
      \put( 0.0, 4.2){Smooth Space $\wtilde U_{α}$}
      \put( 0.0, 0.2){\includegraphics[height=3.5cm]{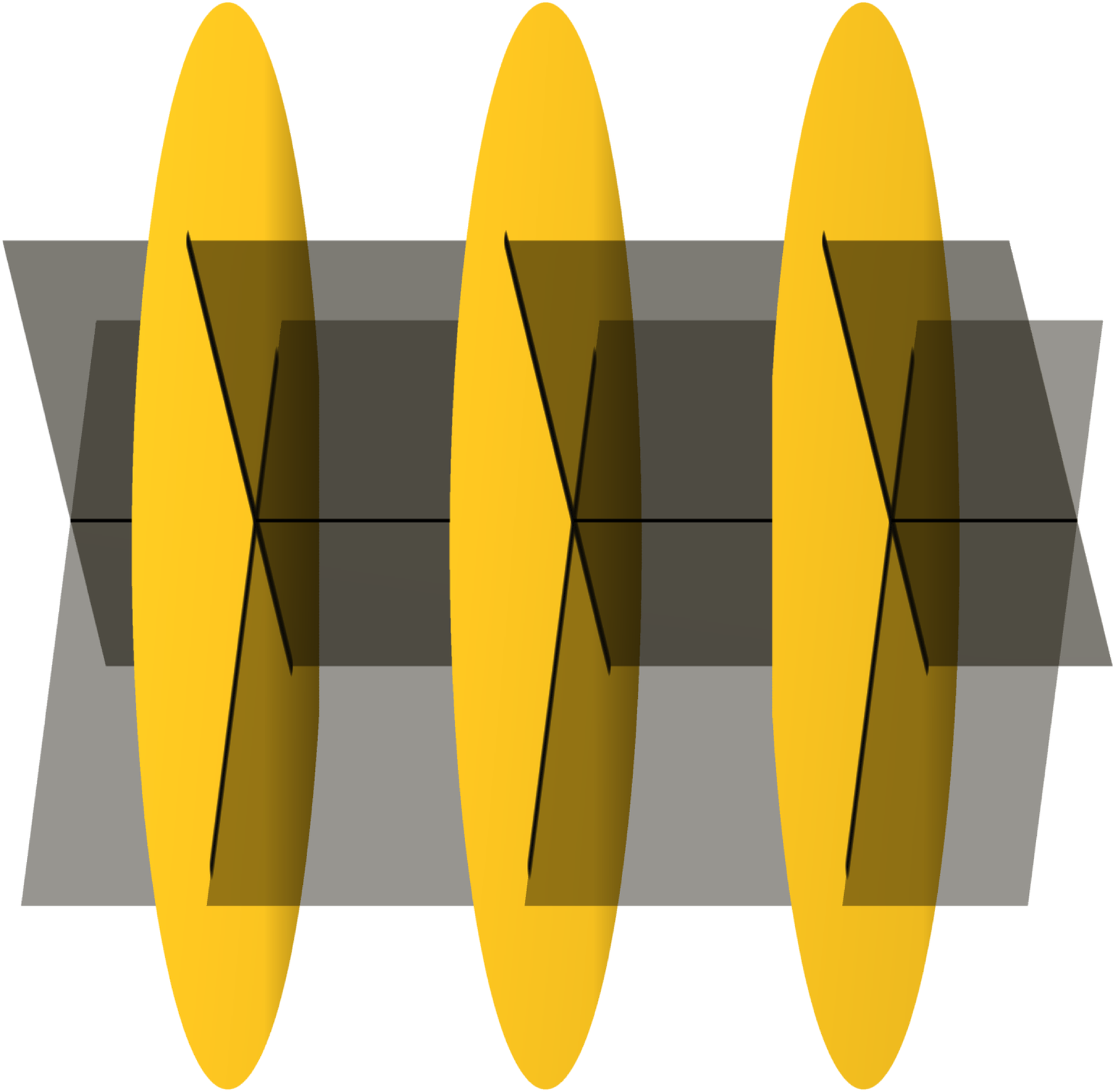}}
      \put( 3.6, 3.5){\scriptsize Divisor $E_0$}
      \put( 3.6, 3.4){\vector(-1, -1){0.4}}
      \put( 4.0, 3.2){\scriptsize Divisor $E_1$}
      \put( 4.0, 3.1){\vector(-1, -1){0.4}}
    \end{picture}}
    \ar@<-.5mm>[rrr]^{π_{α}}_{\text{Resolution Map}} \ar@(d,l)[ddrrr]_(.6){ψ_{α}} &&&
    *++{\begin{picture}(4,4)(0,0)
      \put( 0.0, 4.2){Singular Space $U_{α}$}
      \put( 0.0, 0.2){\includegraphics[height=3.5cm]{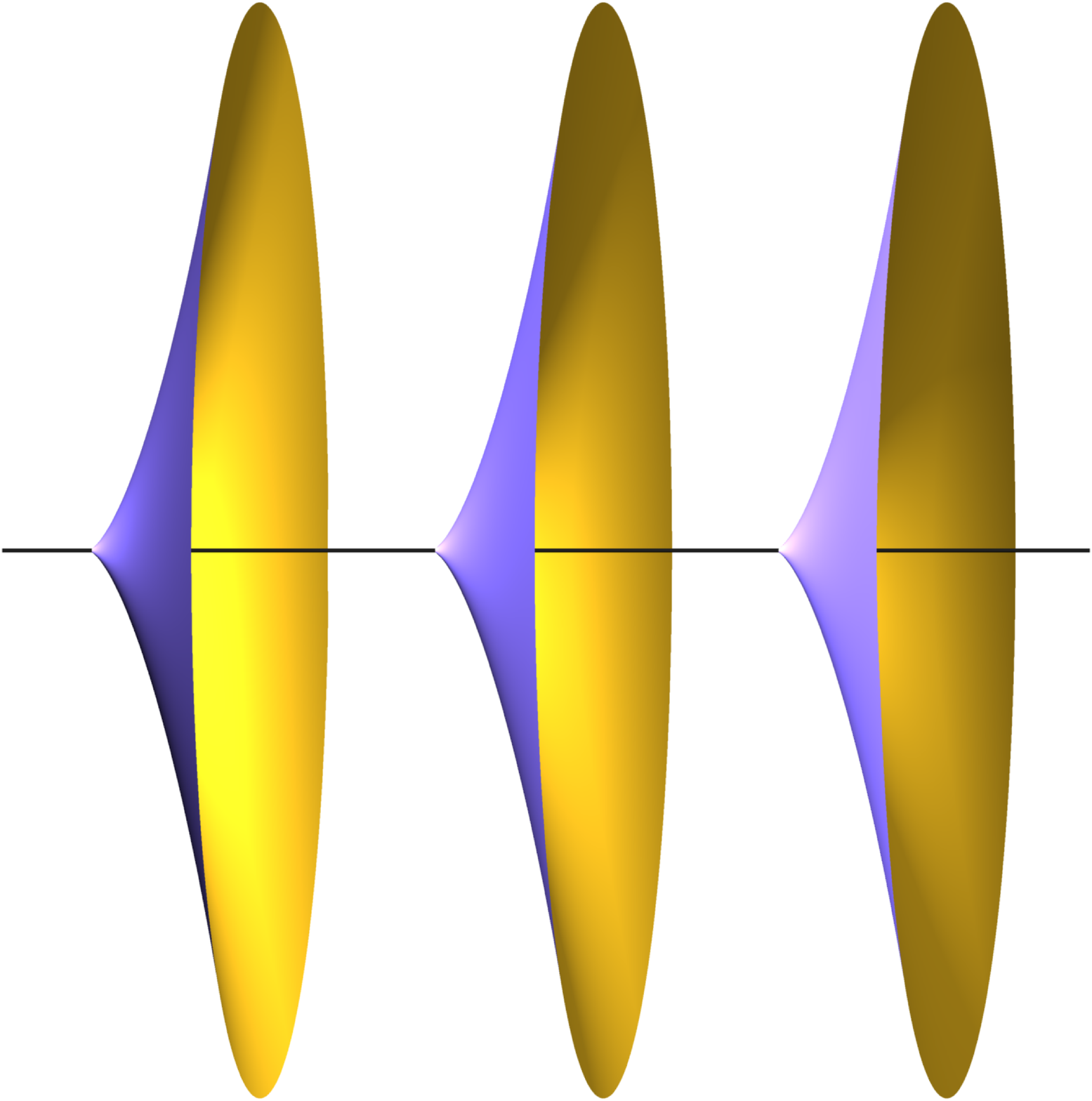}}
      \put( 3.8, 1.3){\scriptsize Curve $T_{α}$}
      \put( 3.8, 1.5){\vector(-1, 1){0.4}}
    \end{picture}}
    \ar[dd]_(.7){φ_{α}}^(.7){\text{Projection to $T_{α}$}} \\\\
    &&&
    *++{\begin{picture}(4,0.6)(0,0)
      \put( 0.0, 0.5){\includegraphics[width=3.5cm]{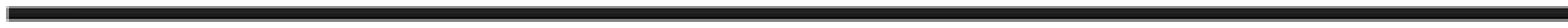}}
      \put( 0.0, 0.0){Curve $T_{α}$}
    \end{picture}}
  }
  $$

  \bigskip

  {\small The figure sketches the situation of Lemma~\ref{lem:proj} in case
    where $X$ is a threefold and $T$ is a curve.  Over $T_{α}$, the
    exceptional set of the resolution map $π$ is a reducible divisor, $E = E_0
    ∪ E_1$.  The composed map $ψ_{α}$ is an snc morphism of the pair
    $(\wtilde U_{α}, E_{α})$.}

  \caption{Projection to general points of $T$}
  \label{fig:TSWAH}
\end{figure}
}{}

\begin{lemNot}[Projection to general points of $T$]\label{lem:proj}
  In the setup of Section~\ref{sssec:PBAnotation}, there exists a dense,
  Zariski-open set $X° ⊆ X$, an open covering of $T° := T ∩ X°$ by sets $U_{α} ⊆
  X°$ that are open in the analytic topology, and commutative diagrams of
  holomorphic morphisms,
  $$
  \xymatrix{
    && \wtilde U_{α} \ar[rr]^(.4){\wtilde{ι}_{α} \text{, inclusion}} \ar@/_4mm/[lld]_{ψ_{α}} \ar[d]_{π|_{\wtilde U_{α}}} && \wtilde X° ⊆ \wtilde X \ar@<4.5mm>[d]^{π}  \ar@<-4.5mm>[d]_{π|_{\wtilde X°}}\\
    T_{α} && U_{α} \ar[ll]^{φ_{α}} \ar[rr]_(.4){ι_{α} \text{, inclusion}} && X° ⊆ X 
  }
  $$
  where
  $$
  \wtilde X° := π^{-1}(X°), \quad \wtilde U_{α} := π^{-1}(U_{α}), \quad T_{α} :=
  T ∩ U_{α},
  $$
  such that the following extra conditions hold
  \begin{enumerate}
  \item\label{il:Qx1} The variety $T°$ is smooth and not empty.  The sheaf
    $Ω^1_{T°}$ is trivial.

  \item\label{il:Qx2} Setting $E° := E ∩ \wtilde X°$, we have $(π°)^{-1}(T°) =
    E°$.

  \item\label{il:Qx3} If $E°_j ⊆ E°$ is any irreducible component, then
    $π°|_{E°_j} : E°_j → T°$ is smooth.

  \item\label{il:Qx4} The restrictions $φ_{α}|_{T_{α}} : T_{α} → T_{α}$ are the
    identity morphisms.

  \item\label{il:Qx5} The holomorphic maps $ψ_{α}$ are smooth.  The divisors
    $E_{α} := E ∩ \wtilde U_{α}$ are relatively snc over $T_{α}$.
  \end{enumerate}
\end{lemNot}
\begin{proof}
  The existence of an open set $X°$ satisfying
  Conditions~(\ref{lem:proj}.\ref{il:Qx1})--(\ref{lem:proj}.\ref{il:Qx3}) is
  clear.  Shrinking $X°$ further, if necessary, the standard technique of
  ``projection to a subvariety'', explained and proven in
  \cite[Prop.~2.26]{GKKP11}, yields the existence of a finite, étale covering
  map $γ: Y → X°$ and a commutative diagram,
  $$
  \xymatrix{
    && \wtilde Y \ar[rrr]^(.4){\wtilde{γ}\text{, finite, étale}} \ar@/_4mm/[lld]_{ψ} \ar[d]_{p} &&& \wtilde X° ⊆ \wtilde X \ar@<4.5mm>[d]^{π}  \ar@<-4.5mm>[d]_{π|_{\wtilde X°}}\\
    B && Y \ar[ll]^{φ} \ar[rrr]_(.4){γ\text{, finite, étale}} &&& X°
    ⊆ X, }
  $$
  where $\wtilde Y := \wtilde X° ⨯_{X°} Y$, where $B := γ^{-1}(T°)$, and where
  $p$ and $\wtilde{γ}$ denote the obvious projections.  The morphism $φ$ has the
  additional property that its restriction to $B ⊂ Y$ is the identity map.
  Shrinking $X°$ further, if necessary, generic smoothness of morphisms and its
  logarithmic analogue, \cite[Rem.~2.11]{GKKP11}, allow to assume that $ψ$ is
  smooth, and that the divisor $E° ⨯_{X°} Y \subsetneq \wtilde Y$ is relatively
  snc over $B$.

  To end the proof, it suffices to find a covering of $T° ⊂ X°$ by analytically
  open sets $(U_{α})_{α ∈ A} ⊆ X°$ that are small enough so that their preimages
  are disjoint unions of $(\deg γ)$-many open sets,
  $$
  γ^{-1} ( U_{α}) = V_{α,1} \, \overset{\cdot}{∪} \cdots
  \overset{\cdot}{∪} \, V_{α, \deg γ},
  $$
  each canonically identified with $U_{α}$.  Choosing one $V_{α, \bullet}$ for
  each given set $U_{α}$, the morphisms $φ_{α}$ and $ψ_{α}$ are immediately
  obtained from these identifications.  Conditions~(\ref{lem:proj}.\ref{il:Qx4})
  and (\ref{lem:proj}.\ref{il:Qx5}) follow from the properties of $φ$ and $ψ$.
\end{proof}

\subsubsection{Proof of Lemma~\ref*{lem:const}: construction of a differential form on $T°$}

We follow the ideas outlined in Section~\ref{sssec:idea}.  Given any point $t ∈
T_{α}$, a fundamental result of Hacon-McKernan asserts that the fibre $E_t :=
π^{-1}(t)$ is rationally chain connected.  Using the results obtained in
Section~\ref{sec:tfdrcc}, this implies that relative differentials in
$Ω^p_{\wtilde U_{α} / T_{α}}$ vanish modulo torsion when restricted to any
component of $E_t$.  This is a first indication of the principle that ``the
restriction of any differential to $E°$ comes from $T°$'', as formulated and
proven in the following lemma.

\begin{lem}[Restriction of any differential to $E°$ comes from $T°$]\label{lem:rDwXc1}
  In the setup of Section~\ref{sssec:proj}, the pull-back map of torsion-free
  differentials,
  $$
  \dnoTor (π|_{E°}) : H^0 \bigl( T°,\, \widecheck{Ω}^p_{T°}
  \bigr) → H^0 \bigl( E°,\, \widecheck{Ω}^p_{E°} \bigr),
  $$
  is isomorphic.
\end{lem}
\begin{proof}
  Choosing an open cover of $T°$ as in Lemma~\ref{lem:proj}, it suffices to show
  that the pull-back maps associated with the restricted morphisms,
  $$
  \dnoTor (π|_{E_{α}}) : H^0 \bigl( T_{α},\, \widecheck{Ω}^p_{T_{α}} \bigr) →
  H^0 \bigl( E_{α},\, \widecheck{Ω}^p_{E_{α}} \bigr),
  $$
  are isomorphisms, for all $α ∈ A$.  Let $α ∈ A$ be any given index.  The
  pull-back map $\dnoTor (π|_{E_{α}})$ is clearly injective.  To prove
  surjectivity, consider the filtration and the sequences introduced in
  Proposition~\ref{prop:TFDrelFilt},
  \begin{gather}
    \widecheck{Ω}^p_{E_{α}} = \widecheck \sF^0 ⊇ \widecheck \sF^1 ⊇
    \cdots ⊇ \widecheck \sF^p ⊇ \widecheck \sF^{p+1} = 0\\
    \label{seq:filt} 0 → \widecheck \sF^{r+1} → \widecheck \sF^r → \left(
      ψ_{α}^* \widecheck{Ω}^r_{T_{α}} \right) \otimes
    \widecheck{Ω}^{p-r}_{E_{α}/T_{α}} → 0,
  \end{gather}
  which exist for all $0 ≤ r ≤ p$.  Observe that in case where $r=p$,
  Sequence~\eqref{seq:filt} yields $\widecheck \sF^p = ψ^*_{α}
  \widecheck{Ω}^p_{T_{α}}$.  To prove Lemma~\ref{lem:rDwXc1}, we aim to show
  that
  \begin{equation}\label{eq:ppo}
    H^0 \bigl( {E_{α}},\, \widecheck{Ω}^p_{E_{α}} \bigr) \overset{!}{=}
    H^0 \bigl( {E_{α}},\, \widecheck \sF^p \bigr) =
    H^0 \bigl( {E_{α}},\, ψ_{α}^* \widecheck{Ω}^p_{T_{α}} \bigr) =
    H^0 \bigl( {T_{α}},\, \widecheck{Ω}^p_{T_{α}} \bigr).
  \end{equation}
  Since the vector bundles, $ψ_{α}^* \widecheck{Ω}^r_{T_{α}} = ψ_{α}^*
  Ω^r_{T_{α}}$ are trivial by Assertion~(\ref{lem:proj}.\ref{il:Qx1}) of
  Lemma~\ref{lem:proj}, Equation~\eqref{eq:ppo} will follow from an inductive
  argument using the Sequences~\eqref{seq:filt} once we show that
  \begin{equation}\label{eq:ppo2}
    H^0 \bigl( E_{α},\, \widecheck{Ω}^q_{E_{α}/T_{α}} \bigr)
    \overset{!}{=} 0 \text{\quad for all $q>0$}.
  \end{equation}
  Using Corollary~\ref{cor:ntorrest} to identify the restriction
  $\widecheck{Ω}^q_{E_{α}/T_{α}} \bigl|_{E_t}$ with $\widecheck{Ω}^q_{E_t}$ for
  all points $t ∈ T_{α}$ and all fibres $E_t := π^{-1}(t) ⊂ E_{α}$,
  Equation~\eqref{eq:ppo2} will in turn follow from the stronger claim that
  \begin{equation}\label{eq:ppq}
    H^0 \bigl( E_t,\, \widecheck{Ω}^q_{E_t} \bigr) \overset{!}{=} 0, \quad \text{for all $t ∈ {T_{α}}$ and all $q>0$.}
  \end{equation}
  To prove~\eqref{eq:ppq}, recall that $X°$ is a klt base space.  A fundamental
  result of Hacon-McKernan \cite[Cor.~1.5(1)]{HMcK07} thus implies that fibres
  $E_t$, being fibres of the birational resolution map $π$, are rationally chain
  connected.  The vanishing asserted in~\eqref{eq:ppq} is therefore an immediate
  consequence of the non-existence of torsion-free forms in rationally chain
  connected spaces, as asserted in Theorem~\ref{thm:diffRCspaces}.
\end{proof}

\begin{rem}
  Using the standard fact that klt base spaces have rational singularities, it
  might be possible to give a proof of Lemma~\ref{lem:rDwXc1} using Namikawa's
  analysis of mixed Hodge structures,
  \cite[Lemma~1.2]{NamikawaDeformationTheory}, rather than the more elementary
  Theorem~\ref{thm:diffRCspaces}
\end{rem}

\begin{cor}[Restriction of $\wtilde{σ}$ to $E°$ comes from a form $η°$ on $T°$]
  In the setup of Section~\ref{sssec:proj}, there exists a unique torsion-free
  form $η° ∈ H^0 \bigl( T°,\, \widecheck{Ω}^p_{T°} \bigr)$ such that
  \begin{equation}\label{eq:bpZ}
    \dnoTor  ι° (\wtilde{σ}) = \dnoTor (π|_{E°}) (η°),
  \end{equation}
  where $ι° : E° → \wtilde X°$ denotes the obvious inclusion.  \qed
\end{cor}

\subsubsection{Proof of Lemma~\ref*{lem:const}: construction of a differential form on an open set of $Z$}

Consider the non-empty, Zariski-open set $Z° := f^{-1}(T°) ⊆ Z$.  We obtain a
torsion-free form
\begin{equation}\label{eq:deftau}
  τ° := \bigl( \dnoTor(f|_{Z°}) \bigr) (η°) ∈ H^0 \bigl( Z°,\, \widecheck{Ω}^{p}_{Z°} \bigr).
\end{equation}
The following elementary lemma summarises what we know about $τ°$.  In essence,
it shows that the form $τ°$ satisfies Equation~\eqref{eq:const2} on the open set
$Z°$.

\begin{lem}[The form $τ°$ satisfies Equation~\eqref{eq:const2} on the open set $Z°$]\label{lem:bpX}
  Setting as above.  Given a smooth variety $V$ and morphisms $a$, $g$ forming a
  commutative diagram as in \eqref{eq:const} of Proposition~\ref{prop:const},
  \PreprintAndPublication{
    \begin{equation}\label{eq:constrep}
      \begin{split}
        \xymatrix{ %
          &&&& \wtilde X \ar[d]_{π}^{\txt{\scriptsize resolution of\\\scriptsize singularities}} \\
          V \ar@/^5mm/[rrrru]^{a} \ar[rr]_{g\text{, surjective}} && Z_{\reg} \ar[rr]_{f|_{Z_{\reg}}} && X,
        }
      \end{split}
    \end{equation}
  }{} set $V° := g^{-1}(Z°)$.  Then
  \begin{equation}\label{eq:bpX}
    \bigl(\dnoTor (a|_{V°}) \bigr)(\wtilde{σ}) = \big( \dnoTor (g|_{V°}) \bigr)(τ°).
  \end{equation}
\end{lem}
\begin{proof}
  It follows from commutativity of
  \PreprintAndPublication{Diagram~\eqref{eq:constrep}}{Diagram~\eqref{eq:const}}
  and from Assumption~\ref{aass:sing} that the image of the restricted morphism
  $a|_{V°}$ is contained in the divisor $E° ⊂ \wtilde X$.  In other words, there
  exists a factorisation
  \begin{equation}\label{eq:bpY}
    \xymatrix{ %
      V° \ar[rr]_{b} \ar@/^4mm/[rrrr]^{a|_{V°}} && E° \ar[rr]_{ι°\text{, inclusion}} && X°.
    }
  \end{equation}
  Equation~\eqref{eq:bpX} then follows easily from the composition law for
  pull-back of torsion-free differentials formulated in Lemma~\ref{lem:clTFD}
  and from the commutativity of the diagrams considered so far.
  \PreprintAndPublication{
    \begin{align*}
      \big( \dnoTor  (g|_{V°}) \bigr)(τ°) & = \bigl(\dnoTor (g|_{V°}) ◦ \dnoTor (f|_{Z°}) \bigr)(η°) && \text{Definition of $τ°$ in \eqref{eq:deftau}} \\
      & = \bigl(\dnoTor b ◦ \dnoTor (π|_{E°}) \bigr)(η°) && \text{Commutativity of~\eqref{eq:const}} \\
      & = \bigl(\dnoTor b ◦ \dnoTor ι° \bigr)(\wtilde{σ}) && \text{Equation~\eqref{eq:bpZ}} \\
      & = \bigl(\dnoTor (a|_{\wtilde V°}) \bigr)(\wtilde{σ}) && \text{Factorisation~\eqref{eq:bpY}.}
    \end{align*}
    This finishes the proof of Lemma~\ref{lem:bpX}.
  }{}
\end{proof}

\begin{rem}[Viewing $τ°$ as a Kähler differential]\label{rem:FGH}
  The varieties $V$, $V°$, $Z$ and $Z°$ are smooth by
  Assumption~\ref{aass:smooth}.  The spaces $H^0 \bigl( Z°,\,
  \widecheck{Ω}^{p}_{Z°} \bigr)$ and $H^0 \bigl( Z°,\, Ω^{p}_{Z°} \bigr)$
  therefore agree, allowing us to view the torsion-free form $τ° ∈ H^0 \bigl(
  Z°,\, \widecheck{Ω}^{p}_{Z°} \bigr)$ as a Kähler differential, say
  $τ°_{\kahler} ∈ H^0 \bigl( Z°,\, Ω^{p}_{Z°} \bigr)$.  Since $a_{V°}$ and
  $g|_{V°}$ are morphisms between smooth spaces, the pull-back notions for
  torsion-free forms and Kähler differentials agree.
  Equation~\eqref{eq:bpXKahler} therefore implies the following equality of
  Kähler differentials
  \begin{equation}\label{eq:bpXKahler}
    \bigl(\dK (a|_{V°}) \bigr)(\wtilde{σ}) = \big( \dK (g|_{V°}) \bigr)(τ°_\kahler).
  \end{equation}
\end{rem}

\subsubsection{Proof of Lemma~\ref*{lem:const}: construction of a differential form on $Z$}

We show that $τ°_\kahler$ can be extended from the open set $Z°$ to all of $Z$.
The following criterion, essentially shown in \cite{GKK08}, will be employed.

\begin{lem}[Regularity criterion for differential forms]\label{lem:regCrit}
  Setup as above.  If there exists a smooth, irreducible variety $\wtilde V$, a
  proper, generically finite surjective morphism $\wtilde g: \wtilde V → Z$ and
  a differential form $\wtilde{τ} ∈ H^0 \bigl( \wtilde V,\, Ω^p_{\wtilde V}
  \bigr)$ whose restriction to $\wtilde V° := \wtilde g^{-1}(Z°)$ agrees with
  the pull-back of $τ°_\kahler$,
  $$
  \bigl( \dK (g|_{V°}) \bigr) (τ°_\kahler) = \wtilde{τ}|_{V°},
  $$
  then there exists a Kähler differential $τ_\kahler ∈ H^0 \bigl( Z,\, Ω^p_Z
  \bigr)$ that extends $τ°_\kahler$.
\end{lem}
\begin{proof}
  Recalling from Assumption~\ref{aass:smooth} that $Z$ is smooth, the sheaf
  $Ω^p_Z$ is locally free, thus reflexive.  We are therefore free to shrink $Z$
  by removing a suitable set of codimension two, and assume without loss of
  generality that $\wtilde g$ is actually finite.  For finite morphisms between
  smooth spaces, Lemma~\ref{lem:regCrit} has already been shown in
  \cite[Cor.~2.12(ii)]{GKK08}.
\end{proof}

\begin{cor}[Regularity of $τ$]
  Setting as above.  Then there exists a Kähler differential $τ_\kahler ∈ H^0
  \bigl( Z,\, Ω^p_Z \bigr)$ that extends $τ°_\kahler$.
\end{cor}
\begin{proof}
  Recall from Remark~\ref{rem:exV} there always exist smooth varieties $V$ and
  morphisms $a$, $g$ forming a Diagram as in \eqref{eq:const}, where $g$ is
  generically finite.  Setting $V° := g^{-1}(Z°)$ as before,
  Equation~\eqref{eq:bpXKahler} of Remark~\ref{rem:FGH} asserts that the
  differential form $\bigl( \dK (g|_{V°}) \bigr)(τ°_\kahler)$ initially defined
  only on $V°$ extends to give a regular differential form on all of $V$, namely
  $\bigl(\dK (a|_{V°}) \bigr)(\wtilde{σ})$.  Lemma~\ref{lem:regCrit} therefore
  applies.
\end{proof}

\subsubsection{Proof of Lemma~\ref*{lem:const}: end of proof}
\label{sssec:polconstend}

Given a smooth variety $V$ and morphisms $a$, $g$ forming a Diagram as in
\eqref{eq:const}, we need to show that Equation~\eqref{eq:const2} holds.  We
have seen in Lemma~\ref{lem:bpX} and Remark~\ref{rem:FGH} that the equation
holds on $Z°$.  Since $Z$ is smooth by Assumption~\ref{aass:smooth}, the sheaf
$Ω^p_Z$ is torsion-free.  Since $Z°$ is open in Zariski topology, hence dense,
Equation~\eqref{eq:const2} holds everywhere, as required.  This finishes the
proof of Lemma~\ref{lem:const}.  \qed

\subsection{Proof of Proposition~\ref*{prop:const}}
\label{ssec:popropconst}

To prove Proposition~\ref{prop:const}, assume that we are given a morphism of
klt base spaces, $f: Z → X$, and a reflexive $p$-form on $X$, say $σ ∈ H^0\bigl(
X,\, Ω^{[p]}_X \bigr)$.

\subsubsection{Proof of Proposition~\ref*{prop:const}: Existence of the form $τ$}

The weak version of Proposition~\ref{prop:const} shown in Lemma~\ref{lem:const}
asserts the existence of a resolution map $π: \wtilde X → X$ and a reflexive
$p$-form $τ ∈ H^0\bigl( Z,\, Ω^{[p]}_Z \bigr)$ satisfying universal property
\eqref{eq:const2}.  To prove Proposition~\ref{prop:const}, we will show that the
form $τ$ satisfies the stronger requirements of Proposition~\ref{prop:const}.
To this end, assume we are given an arbitrary resolution of singularities $π' :
\wtilde X' → X$ with form $\wtilde{σ}' ∈ H^0 \bigl( \wtilde X',\, Ω^p_{\wtilde
  X'} \bigr)$, a smooth variety $V'$ and morphisms $a'$, $g'$ forming a
commutative diagram as in~\eqref{eq:const}.  We need to show that the following
equation holds,
\begin{equation}\label{eq:const4}
  \dK a'(\wtilde{σ}') = \dK g'(τ|_{Z_{\reg}}) ∈ H^0 \bigl( V,\, Ω^p_V \bigr).
\end{equation}
To this end, choose a component of the fibred product $\wtilde X' ⨯_X \wtilde X$
that surjects onto $X$, and let $\what X$ be a desingularisation of that
component.  We obtain a resolution of singularities $\what{π} : \what X → X$
that dominates both $\wtilde X_1$ and $\wtilde X_1$.  In a similar vein, choose
a component of $V' ⨯_X \what X$ that surjects onto $V'$.  Desingularising, we
obtain a smooth variety $V$ and a large commutative diagram of morphisms between
varieties,
$$
\xymatrix{ %
  &&&&&& \what X \ar[d]_{p'} \ar@{=}[r] & \what X \ar[d]_{p} \ar@/^4mm/[dd]^{\what{π}}\\
  &&&&&& \wtilde X' \ar[d]_{π'} & \wtilde X \ar[d]_{π}\\
  V \ar@/^6mm/[rrrrrruu]^{a''} \ar[rr]_{g''\text{, surjective}} && V' \ar@/^3mm/[rrrru]^{a'} \ar[rr]_{g'\text{, surjective}} && Z_{\reg} \ar[rr]_{f|_{Z_{\reg}}} && X \ar@{=}[r]& X.   }
$$
Observing that the pull-back of the forms $\wtilde{σ}$ and $\wtilde{σ}'$ to the
common resolution $\what X$ agree on the preimage of $X_{\reg}$, we obtain an
equality of differential forms,
\begin{equation}\label{eq:pbthX}
  \dK p'(\wtilde{σ}') = \dK p(\wtilde{σ}).
\end{equation}
An application of Lemma~\ref{lem:const} with $a=p ◦ a''$, $g = g' ◦ g''$
therefore yields the following chain of equalities,
\begin{align*}
  \bigl( \dK g'' ◦ \dK g' \bigr) (τ) & = \bigl( \dK a'' ◦  \dK p \bigr)  (\wtilde{σ}) && \text{Lemma~\ref{lem:const}}\\
  & =  \bigl( \dK a'' ◦ \dK p' \bigr) (\wtilde{σ}') && \text{Equation~\eqref{eq:pbthX}}\\
  & =  \bigl( \dK g'' ◦ \dK a' \bigr) (\wtilde{σ}') && \text{Equality $p' ◦ a'' = a' ◦ g''$}
\end{align*}
Equality~\eqref{eq:const4} now follows because $g''$ is surjective and $\dK g''$
therefore injective.  We have thus shown that $τ$ satisfies the requirements of
Proposition~\ref{prop:const}.

\subsubsection{Proof of Proposition~\ref*{prop:const}: Uniqueness of the form $τ$}

Remark~\ref{rem:exV} asserts the existence of a diagram as in \eqref{eq:const}.
Choosing one such diagram, it follows from surjectivity of $g$ that the
restriction-and-pull-back map
$$
H^0 \bigl( Z,\, Ω^{[p]}_Z \bigr) → H^0 \bigl( V,\, Ω^p_V \bigr),
\quad μ \mapsto \dK g(μ|_{Z_{\reg}})
$$
is injective.  Equation~\eqref{eq:const2} therefore determines the reflexive form
$τ$ uniquely.  This finishes the proof of Proposition~\ref{prop:const} \qed

\section{Proof of Theorem~\ref*{thm:PB-thmA}}
\label{sec:pfPB}

\subsection{Proof of Theorem~\ref*{thm:PB-thmA}: Uniqueness of the functor}
\label{ssec:pfPBuq}

Assume we are given two functors $\drefl^1$ and $\drefl^2$ that satisfy the
conditions of Theorem~\ref{thm:PB-thmA}.  Given any morphism of klt base spaces,
$f: Z → X$, denote the two associated pull-back maps of reflexive differentials
by
$$
\drefl^1 f : H^0 \bigl( X,\, Ω^{[p]}_X \bigr) → H^0 \bigl( Y,\, Ω^{[p]}_Y \bigr) \text{\quad and \quad}
\drefl^2 f : H^0 \bigl( X,\, Ω^{[p]}_X \bigr) → H^0 \bigl( Y,\, Ω^{[p]}_Y \bigr),
$$
respectively.  To prove Theorem~\ref*{thm:PB-thmA}, we need to show that
\CounterStep
\begin{equation}\label{eq:uniq}
  \drefl^1 f = \drefl^2 f.
\end{equation}
This equality will be established in the remainder of the present
Section~\ref{ssec:pfPBuq}.

\subsubsection{Proof of uniqueness in special cases}

Before proving Equality~\eqref{eq:uniq} in general, we treat two special cases
first.

\begin{lem}\label{lem:X1}
  If there exists a dense open subset $Z° ⊆ Z$ with inclusion $i : Z° → Z$ such
  that
  \begin{equation}\label{eq:XD}
    \drefl^1(f ◦ i) = \drefl^2(f ◦ i),
  \end{equation}
  then Equation~\eqref{eq:uniq} holds.
\end{lem}
\begin{proof}
  By functoriality, Assumption~\eqref{eq:XD} translates as
  $$
  \drefl^1 i ◦ \drefl^1 f = \drefl^2 i ◦ \drefl^2 f.
  $$
  To prove Lemma~\ref{lem:X1}, it will therefore suffice to show that $\drefl^1
  i = \drefl^2 i$, and that these maps are injective.  That, however, follows
  immediately from compatibility with Kähler differentials, as formulated in
  Diagram~\eqref{eq:BCD} of Theorem~\ref{thm:PB-thmA}: given any reflexive form
  $σ ∈ H^0 \bigl( Z,\, Ω^{[p]}_Z \bigr)$, then
  $$
  \bigl( \drefl^1 i \bigr) (σ) \bigr|_{Z°_{\reg}} =
  \bigl( \dK i \bigr) (σ |_{Z°_{\reg}} ) =
  \bigl( \drefl^2 i \bigr) (σ) \bigr|_{Z°_{\reg}}.
  $$  
  Observe that this determines $\bigl( \drefl^\bullet i \bigr) (σ)$ uniquely.
\end{proof}

\begin{lem}\label{lem:X2}
  If $f$ is surjective, then Equation~\eqref{eq:uniq} holds.
\end{lem}
\begin{proof}
  It is clear by assumption that the open set $Z° := Z_{\reg} ∩ f^{-1}(X_{\reg})
  ⊆ Z$ is not empty.  Compatibility with Kähler differentials then asserts that
  $\drefl^1(f ◦ i) = \drefl^2(f ◦ i)$, where $i: Z° → Z$ is the obvious open
  immersion.  Lemma~\ref{lem:X1} therefore applies.
\end{proof}

\subsubsection{Proof of uniqueness in general}

Next, we treat the general case.  To this end, consider a commutative diagram
$$
\xymatrix{ %
  \wtilde Z \ar[r]^{\wtilde f} \ar[d]_{π_Z} & \wtilde X \ar[d]_{π}^{\txt{\scriptsize strong resolution of\\\scriptsize singularities}} \\
  Z \ar[r]_{f} & X }
$$
where $\wtilde Z$ is a desingularisation of a component of the fibre product $Z
⨯_X \wtilde X$ that surjects onto $Z$.  Using Lemma~\ref{lem:X1}, we may replace
$Z$ with a dense open subset and assume without loss of generality that $Z$ is
smooth.  The following identities are now immediate consequences of
compatibility with Kähler differentials.  \CounterStep
\begin{align}
  \label{eq:A1} \drefl^1 π & = \drefl^2π && \text{by Lemma~\ref{lem:X2}}\\
  \label{eq:A2} \drefl^1 \wtilde f & = \drefl^2 \wtilde f  = \dK \wtilde f && \text{by~\eqref{eq:BCD} since $\wtilde Z$ and $\wtilde X$ are smooth}\\
  \label{eq:A3} \drefl^1 π_Z & = \drefl^2 π_Z  = \dK π_Z && \text{by~\eqref{eq:BCD} since $\wtilde Z$ and $Z$ are smooth}
\end{align}
The equalities have several consequences.  First, we see that
\begin{align*}
  \drefl^1 (π ◦ \wtilde f) & = (\drefl^1\wtilde f) ◦ (\drefl^1 π) && \text{functoriality of $\drefl^1$} \\
  & = (\drefl^2\wtilde f) ◦ (\drefl^2 π) && \text{by~\eqref{eq:A1}, \eqref{eq:A2} } \\
  & = \drefl^2 (π ◦ \wtilde f) && \text{functoriality of $\drefl^2$}.
\end{align*}
Since $π ◦ \wtilde f = f ◦ \wtilde{π}_Z$, this immediately implies that
\begin{equation}\label{eq:GFK}
  (\drefl^1 π_Z) ◦ (\drefl^1 f) = (\drefl^2 π_Z) ◦ (\drefl^2 f)  
\end{equation}
But since $\drefl^1π_Z$ and $\drefl^2π_Z$ are both equal to the standard
pull-back of Kähler differentials, and since $π_Z$ is surjective, it is clear
that two forms $σ$, $τ$ on $Z$ agree if and only if $\drefl^1 π_Z (σ) = \drefl^2
π_Z (τ)$.  Equation~\eqref{eq:GFK} therefore implies Assertion~\eqref{eq:uniq},
finishing the proof of the uniqueness statement in Theorem~\ref{thm:PB-thmA}.
\qed

\subsection{Proof of Theorem~\ref*{thm:PB-thmA}: Existence of the functor}

Given a morphism $f : Z → X$ of klt base spaces, we define an associated
pull-back mapping of reflexive differential forms,
$$
\drefl f : H^0 \bigl( X,\, Ω^{[p]}_X \bigr) → H^0 \bigl( Z,\,
Ω^{[p]}_Z \bigr),
$$
by sending a given reflexive form $σ ∈ H^0 \bigl( X,\, Ω^{[p]}_X \bigr)$ to the
unique form $τ ∈ H^0 \bigl( Z,\, Ω^{[p]}_Z \bigr)$ whose existence is asserted
in Proposition~\ref{prop:const}.  We need to show that the so-defined $\drefl$
satisfies the composition law.  To this end, consider a sequence of morphisms
between klt base spaces, say $f' : Z' → Z$ and $f: Z → X$.  We need to show that
\CounterStep
\begin{equation}\label{eq:Aa}
  \bigl( \drefl f' ◦ \drefl f \bigr)(σ) \,\,\stackrel{!}{=}\,\, \bigl( \drefl(f ◦ f') \bigr) (σ) \quad
  \text{for all } σ ∈ H^0 \bigl( X,\, Ω^{[p]}_X \bigr).
\end{equation}
To prove~\eqref{eq:Aa}, choose desingularisation $p: \wtilde Z → Z$, choose a
component of $\wtilde Z ⨯_X \wtilde X$ that surjects onto $\wtilde Z$, and let
$\what Z$ be a desingularisation of this component.  Further, choose a component
of $\what Z ⨯_Z Z'$ that surjects onto $Z'$, and let $V$ be a desingularisation
of this component.  In summary, we have constructed smooth spaces $\wtilde Z$,
$\what Z$ and $V$ fitting into a commutative diagram
\begin{equation}\label{eq:xxA}
  \begin{split}
    \xymatrix{ %
      &&&& \what Z \ar[rr]^{h} \ar[d]_{q}^{\txt{\scriptsize surjective}} && \wtilde X \ar[dd]_{π}^{\txt{\scriptsize resolution of\\\scriptsize singularities}}\\
      &&&& \wtilde Z \ar[d]_{p}^{\txt{\scriptsize resolution of\\\scriptsize singularities}} \\
      V \ar[rr]_{g\text{, surjective}} \ar@/^6mm/[rrrruu]^{\what a} \ar@/^4mm/[rrrru]_a && Z' \ar[rr]_{f'} && Z \ar[rr]_f && X.
    }
  \end{split}
\end{equation}
The following reflexive forms, defined as explained above using
Proposition~\ref{prop:const}, will appear in the computation
$$
\wtilde{σ} := \drefl π(σ) \qquad τ := \drefl f(σ) \qquad
\wtilde{τ} := \drefl p(τ).
$$
We have seen in Remark~\ref{rem:exV2} that this definition implies that
$\wtilde{σ}$ is the unique differential form on $\wtilde X$ that agrees with the
pull-back $σ$ at points where $π$ is isomorphic.  The analogue statement holds
for $\wtilde{τ}$ and $τ$, so that our notation is consistent with the notation
used earlier.  A repeated application of Proposition~\ref{prop:const}, using the
fact that two reflexive forms on a normal space agree if they agree on the
smooth locus, now shows the following
\begin{align*}
  \drefl g \bigl( (\drefl f' ◦ \drefl f)(σ) \bigr) & =  \drefl g \bigl( \drefl f'(τ)\bigr)  && \text{Definition of $τ$}\\
  & = \drefl a(\wtilde{τ}) && \text{Proposition~\ref{prop:const} for $f'$} \\
  & = \bigl( \drefl \what a ◦ \drefl q \bigr) (\wtilde{τ}) && \text{Remark~\ref{rem:exV2}, Diag.~\eqref{eq:xxA}}\\
  & = \bigl( \drefl \what a ◦ \drefl h \bigr) (\wtilde{σ}) && \text{Proposition~\ref{prop:const} for $f$}\\
  & = \drefl g \bigl( \drefl (f◦ f') (\wtilde{σ}) \bigr) && \text{Proposition~\ref{prop:const} for $f ◦ f'$}.
\end{align*}
Since $g$ is surjective, Remark~\ref{rem:exV2} immediately implies that $\drefl
g$ is injective.  Equation~\eqref{eq:Aa} therefore follows from the computation.
In summary, we have shown that the definition of pull-back given above does
satisfy the composition law.  This finishes the proof of the existence statement
in Theorem~\ref{thm:PB-thmA}.  \qed

\part{Appendix}
\appendix

\section{Torsion sheaves on reducible spaces}
\label{app:A}

In Parts~\ref{part:1} and \ref{part:2}, we need to discuss torsion sheaves and
torsion-free sheaves on reducible spaces.  While no fundamental issues arise, it
seems that almost all standard books, such as \cite{Ha77}, \cite{EGA1} or
\cite{CAS} restrict themselves to the irreducible case.  The few existing
references touch the subject only very briefly.  For completeness' sake, we have
thus chosen to recall the relevant definitions and to include proofs of all the
properties used in this paper.

\subsection{The definition of torsion sheaves}

We briefly recall the definition of torsion sheaves given in
\cite[I.8]{EGA1-Reedition}\footnote{The definition presented here is found in
  \cite{EGA1-Reedition} but not in \cite{EGA1}.  At the time of writing this
  paper, the book \cite{EGA1-Reedition} was not listed on MathSciNet and did not
  show on \url{www.springer.com}.}, see also \cite[§20.1]{EGA4-4}.

\begin{notation}[\protect{Sheaf of rational functions, \cite[I.8.3]{EGA1-Reedition}}]\label{not:ratfct}
  Let $X$ be a reduced, quasi-projective scheme.  We denote the sheaf of
  rational functions on $X$ by $\sR_X$.
\end{notation}

\begin{explanation}
  In the setting of Notation~\ref{not:ratfct}, the sheaf of rational functions
  is quasi-coherent.  If $X° ⊆ X$ is an affine open set, say $X° = \Spec A$,
  then $\sR_X(X°)$ is the ring of rational functions on $X$.  This ring is
  isomorphic to the localisation $S^{-1} A$, where $S$ is the multiplicatively
  closed set of non-zerodivisors in $A$.
\end{explanation}

\begin{defn}[\protect{Torsion sheaf, \cite[I.8.4]{EGA1-Reedition}}]\label{defn:torsion}
  Let $X$ be a reduced, quasi-projective scheme and $ψ: \sO_X → \sR_X$ the
  natural inclusion of the structure sheaf into the sheaf of rational functions.
  Given a coherent sheaf $\sF$ of $\sO_X$-modules consider the natural map
  $ψ_{\sF}$ given as the composition of the following maps,
  $$
  \sF \xrightarrow{ \;\cong\; } \sF \otimes_{\sO_X} \sO_X
  \xrightarrow{\;\Id_{\sF} \otimes_{\sO_X} ψ\;} \sF \otimes_{\sO_X} \sR_X.
  $$
  Define the \emph{torsion subsheaf} of $\sF$ as $\tor \sF := \ker ψ_{\sF}$.
  The sheaf $\sF$ is called \emph{torsion sheaf} if $ψ_{\sF} = 0$, and
  \emph{torsion-free} if $ψ_{\sF}$ is injective.
\end{defn}

\begin{explanation}\label{expl:torgenpts}
  In the setting of Definition~\ref{defn:torsion}, let $X° ⊆ X$ be any affine
  open set, say $X° = \Spec A$.  Denoting the $A$-module associated with the
  sheaf $\sF$ by $F := \sF(X°)$, the sheaf $\sF \otimes_{\sO_X} \sR_X$ of
  Definition~\ref{defn:torsion} is expressed as follows,
  $$
  \sF \otimes_{\sO_X} \sR_X =  \sF^\sim  \otimes_{\sO_X} (S^{-1}A)^\sim = (F  \otimes_A S^{-1}A)^\sim = (S^{-1}F)^\sim.
  $$
  In summary, we see that a section $σ ∈ \sF(X°)$ is a section of the torsion
  subsheaf $\tor \sF$ if and only if there exists a non-zerodivisor $f ∈
  \sO_X(X°)$ that annihilates it.  In particular, $σ$ is a section of the
  torsion subsheaf if and only if there exists dense open $U ⊆ X°$ such that
  $σ|_U = 0$.
\end{explanation}

\begin{rem}[Torsion-free sheaves that are zero on open sets]\label{rem:tfszoos}
  Torsion-free sheaves on reducible spaces can restrict to the zero sheaf on a
  Zariski-open set, as long as the set is not dense.
\end{rem}

\subsection{Elementary properties}

Definition~\ref{defn:torsion} ensures that essentially all properties known from
torsion-free sheaves on irreducible spaces also hold in the more general setting
of reduced quasi-projective schemes ---except perhaps the feature mentioned in
Remark~\ref{rem:tfszoos} above.  The following properties have been used in this
paper.

\begin{prop}[Universal property of torsion-freeness]\label{prop:UPTf}
  Let $X$ be a reduced quasi-projective scheme, and $φ : \sF → \sG$ a morphism
  of coherent sheaves of $\sO_X$-modules.  Then there exists a unique morphism
  $\widecheck{φ}$ making the following diagram commutative,
  $$
  \xymatrix{ %
    \sF \ar[r]^{φ} \ar[d]_{\text{quotient}} & \sG \ar[d]^{\text{quotient}} \\
    \factor \sF.\tor.  \ar[r]_{\widecheck{φ}} & \factor \sG.{\tor.}.  }
  $$
  If $ψ$ is injective, then $\widecheck{φ}$ is injective as well.
\end{prop}
\begin{proof}
  Observing that $\sF/\tor = \Image ψ_{\sF}$ and $\sG/\tor = \Image ψ_{\sG}$,
  the claim is immediate from right-exactness and from universal properties of
  the tensor product.
\end{proof}

\begin{cor}[Subsheaves of torsion-free sheaves are torsion-free]\label{cor:sstf}
  Let $X$ be a reduced quasi-projective scheme, and $\sF$ a torsion-free
  coherent sheaf of $\sO_X$-modules.  If $\sG$ is any subsheaf of
  $\sO_X$-modules, then $\sG$ is likewise torsion-free.  \qed
\end{cor}

\begin{prop}[\protect{Push-forward of torsion-free is torsion-free, compare \cite[I.~Prop.~8.4.5]{EGA1-Reedition}}]\label{prop:pftf}
  Let $X$ be a reduced quasi-projective scheme.  If $ι : X_0 → X$ denotes the
  inclusion map of one irreducible component, and if $\sF$ is a torsion-free
  coherent sheaf on $X_0$, then $ι_* \sF$ is a torsion-free coherent sheaf on
  $X$.
\end{prop}
\begin{proof}
  Let $X° ⊂ X$ be any affine open subset.  If $X°$ is disjoint from $X_0$, then
  $ι_* \sF|_{X°}$ is the zero sheaf, which is torsion-free.

  Now assume that $X° ∩ X_0 \not = \emptyset$, and let
  $$
  σ ∈ (\tor ι_* \sF)(X°) ⊆  (ι_* \sF)(X°)
  $$
  be any section, with associated section $τ ∈ \sF(X° ∩ X_0)$.  We need to show
  that $σ = 0$, or equivalently $τ = 0$.  To this end, let $X_1 ⊂ X$ denote the
  union of all irreducible components different from $X_0$.  Let $U ⊂ X°$ be a
  dense open subset such that $σ|_U = 0$.  The set $U° := U ∩ (X_0 \setminus
  X_1)$ is then open and dense on $X_0$ and $σ|_{U°} = 0$.  In particular,
  $τ|_{U°} = 0$.  Since $\sF$ is torsion-free, this shows that $τ = 0$, as
  claimed.
\end{proof}

\begin{prop}[Injectivity of morphisms]\label{prop:ijtf}
  Let $X$ be a reduced, quasi-projective scheme and $φ : \sF → \sG$ a morphism
  of coherent sheaves of $\sO_X$-modules.  Assume that $\sF$ is torsion-free.
  If $Y \subsetneq X$ is a closed subset such that
  \begin{enumerate}
  \item $Y$ does not contain any irreducible component of $X$, and
  \item the restricted morphism $φ|_{X \setminus Y}$ is injective,
  \end{enumerate}
  then $φ$ is injective.
\end{prop}
\begin{proof}
  Assume we are given an affine open set $X° ⊆ X$ and a section $σ ∈ \ker φ$.
  Observe that $U := X° \setminus Y$ is dense in $X°$, and that $σ|_U = 0$.
  Since $\sF$ is torsion-free this implies that $σ = 0$.
\end{proof}

\providecommand{\bysame}{\leavevmode\hbox to3em{\hrulefill}\thinspace}
\providecommand{\MR}{\relax\ifhmode\unskip\space\fi MR }
\providecommand{\MRhref}[2]{%
  \href{http://www.ams.org/mathscinet-getitem?mr=#1}{#2}
}
\providecommand{\href}[2]{#2}

\end{document}